\documentclass[12pt]{article}
\usepackage[T1]{fontenc}
\usepackage[utf8]{inputenc}
\usepackage[english]{babel}

\usepackage{amsthm,amssymb,amsmath}
\usepackage[margin=3.5cm]{geometry}
\usepackage{graphicx}
\usepackage{refcount, color, soul}
\usepackage{subcaption}

\newtheorem{prop}{Proposition}[subsection]
\newtheorem{theo}[prop]{Theorem}

\newtheorem{Wtheo}[prop]{Weierstrass Preparation Theorem}
\newtheorem{cor}[prop]{Corollary}
\newtheorem{lem}[prop]{Lemma}
\theoremstyle{definition}
\newtheorem{deff}[prop]{Definition}
\newtheorem{notation}[prop]{Notation}
\theoremstyle{remark}
\newtheorem{rem}[prop]{Remark}

\let\joli=\mathcal 
\providecommand\keywords[2]{{\footnotesize\noindent\textbf{Keywords} --- #1%
  \par\noindent 2020 \textit{Mathematics Subject Classification}.\null{} #2}}
\newcommand\ftilde{\tilde f}
\newcommand\gtilde{\tilde g}
\newcommand\Cp{\mathbb{C}} \def\R{\mathbb{R}}
\newcommand\N{\mathbb{N}} \def\Z{\mathbb{Z}}
\newcommand\Ss{\mathbb S}
\newcommand\Pp{\mathbb P}
\newcommand\Opi{\Omega_{\delta}}
\newcommand\inv{^{-1}}

\newcommand\vt{v^{\scriptscriptstyle{1\over 2}}}
\newcommand\Tt{{T_{\scriptscriptstyle{1\over 2}}}}

\newcommand\vun{v^{\scriptscriptstyle1}}
\newcommand\zpoint{{\dot z}}

\newcommand\zbar{{\overline z}}

\newcommand\wbar{{\overline w}}

\let\phi=\varphi
\let\ep=\varepsilon
\let\eps=\varepsilon
\def\epbar{{\overline \ep}}
\def\wep{{\widetilde \ep}}
\newcommand\hatep{{\widehat\ep}}
\newcommand\hatepbar{{\ol\hatep}}

\let\del=\partial

\let\ol=\overline
\newcommand\STt{\Sigma\circ\Tt}
\def\STtt_#1{\Sigma\circ T_{\scriptscriptstyle#1}}
\def\Ttt_#1{{T_{\scriptscriptstyle{#1}}}}
\expandafter\let\expandafter\over\csname @@over\endcsname
\expandafter\let\expandafter\atop\csname @@atop\endcsname

\let\wt=\widetilde

\def\raggedcenter{\leftskip=0pt plus4em \rightskip=\leftskip%
  \parfillskip=0pt \spaceskip=.3333em \xspaceskip=.5em
  \pretolerance=9999 \tolerance=9999 \parindent=0pt
  \hyphenpenalty=9999 \exhyphenpenalty=9999 }

\usepackage{caption}
\usepackage[hidelinks]{hyperref}
\hypersetup{
  pdftitle={Analytic Classification of Germs of Parabolic Antiholomorphic 
    Diffeomorphism of Codimension k},
  pdfauthor={Jonathan Godin, Christiane Rousseau},
  pdfsubject={Complex dynamical systems},
  pdfkeywords={Discrete dynamical systems,
  antiholomorphic dynamics, parabolic fixed point, 
  classification, unfoldings, modulus of analytic classification}
}

\begin{document}
\def\d{{\rm d}}

\title%[Unfoldings of Parabolic Fixed Point of Codimension 1]
  {Analytic Classification of Generic Unfoldings of Antiholomorphic Parabolic Fixed Points of Codimension~1}

\author{Jonathan Godin and Christiane Rousseau\protect\footnote{The first 
  author was supported by a FRQ-NT PhD scholarship. 
  The second author is supported by NSERC in Canada.}}
%\address{Universit\'e de Montr\'eal, C.P.\null{} 6128, Succ.\null{} Centre-Ville,\\
%  Montr\'eal, Qc, Canada, H3C 3J7\\
%  \email{godinj@dms.umontreal.ca, rousseac@dms.umontreal.ca}\\}

\let\Centering=\centering
\let\centering=\raggedcenter
\maketitle
\let\centering=\Centering
\let\Centering=\undefined

\tableofcontents

\newpage

\begin{abstract}
We classify generic unfoldings of germs of antiholomorphic
  diffeomorphisms with a parabolic point of codimension~1 (i.e.~a double fixed point) under conjugacy. These generic unfoldings depend on one real parameter. 
  The classification is done by assigning to each such germ a weak and a strong modulus, which are unfoldings of the modulus assigned to the antiholomorphic parabolic point. The weak and the strong moduli are unfoldings of  the \'Ecalle-Voronin modulus  of the second  iterate of the germ which is a real unfolding of a holomorphic parabolic point. A preparation of the unfolding allows to identify one real analytic \emph{canonical parameter} and any conjugacy between two prepared generic unfoldings preserves the canonical
  parameter. We also solve the realisation problem by giving necessary and sufficient conditions for a strong modulus to be realized. This is done simultaneously with solving the probem of the existence of an antiholomorphic square root to a germ of generic analytic unfolding of a holomorphic parabolic germ. As a second application we establish the condition for the existence of a real analytic invariant curve.
  \end{abstract}
\smallskip
\noindent\keywords{Discrete dynamical systems,
  antiholomorphic dynamics, parabolic fixed point, 
  classification, unfoldings, modulus of analytic classification}
  {37F46 32H50 37F34 37F44}

\clearpage

\section{Introduction}
In this paper, we are interested in unfoldings of 
antiholomorphic germs with a parabolic fixed point of codimension~1,
i.e.~of multiplicity 2. Such an antiholomorphic germ $f_0\colon(\Cp,0)\to (\Cp,0)$ has the form
$$
  f_0(z) = \zbar + {1\over 2}\zbar^2 
    + \left({1\over 4}-{b\over 2}\right) \zbar^3 + o(\zbar^3),
$$
in some local coordinate and for some invariant $b\in\R$.

This study is part of a large program to understand the local dynamics of singularities in low-dimensional complex dynamical systems, with particular emphasis on $1$-resonant singularities (i.e.~all resonance relations among the eigenvalues/multipliers are consequences of a single one). The  simplest case is that of the parabolic point (multiple fixed point) of a holomorphic germ $f_0\colon(\Cp,0)\to (\Cp,0)$. In that case one classifies the germs under conjugacy, i.e.~local analytic changes of coordinate: the classification has been given by \'Ecalle and Voronin, and to each germ is associated its classifying object, its \emph{\'Ecalle-Voronin modulus} (see \cite{V}, \cite{E}, or \cite{nonlinear}, or \cite{lecturesDiffEq}).

Antiholomorphic dynamics has been mainly studied in the context of Julia sets of unicritical antiholomophic polynomials $\bar z^d+c$, where the connectedness loci are given as multicorns:  see \cite{multiI}, \cite{multiNotPath}, \cite{nonlanding}, \cite{multiII}. There it can be seen that the boundary points of hyperbolic components of odd period consist only of parabolic parameter values: this boundary is composed of real analytic arcs corresponding to generic codimension 1 bifurcations of antiholomorphic parabolic points linked at higher codimension isolated points. 
These references generalize some tools from the holomorphic case: Fatou coordinates, Ecalle cylinders. 

The analytic classification of germs of antiholomorphic diffeomorphisms $f_0$ with a parabolic fixed point of arbitrary codimension $k$ is
given in~\cite{GR}. For such a germ $f_0$, the holomorphic germ $g_0=f_0\circ f_0$ has a parabolic fixed point of the same codimension $k$. Hence it is no surprise that a modulus of classification is given by the \'Ecalle-Voronin modulus of $g_0$, namely  a modulus composed
of the codimension, the formal invariant $b$, and $2k$ horn maps. However not all \'Ecalle-Voronin moduli are realisable: indeed $k$ horn maps determine the $k$ other horn maps. Moreover, the formal invariant $b$ is real. 
%One of the key observation to obtain the classication is that the second
%iterate is a holomorphic germ with a parabolic fixed point.
%In fact, a horn map of the modulus of $f_0$ is enough
%to obtain a pair of horn maps that describes the dynamics of $f_0\circ f_0$. 
%With the modulus,
%we see that two ``random'' antiholomorphic germs with a parabolic
%fixed point will, in genral, not be equivalent, i.e.~holomorphically
%conjucate.

In dimension $1$, the $1$-resonant singularities occur as the coalescence of fixed points and/or periodic orbits. Hence it is natural to embed the diffeormorphism in a family of diffeomorphims, called an \emph{unfolding} separating the singularities into simple ones. Then each singularity is organizing rigidly the dynamics in its neighbourhood and the local \lq\lq models\rq\rq\ may not match globally. It is the limit of this mismatch that produces the \'Ecalle-Voronin modulus. This is why it is
natural to study generic unfoldings of $1$-resonant singularities. The  program has been performed for holomorphic parabolic points in codimension $k$ (\cite{germeDeploie},  \cite{realisation} and   \cite{R1}), and for resonant diffeormorphisms (\cite{germeResonant} and \cite{R2}). In this paper we consider the analytic classification of generic unfoldings of codimension $1$ antiholomorphic parabolic points. 

The condition for a fixed point to be parabolic is 
$|f_0'(0)| = 1$, a condition of \emph{real} codimension~1.
Hence it is natural to unfold $f_0$ with one real parameter $\ep$.
An unfolding $f_\ep$ will either have two fixed points,
or a periodic orbit of period 2, or a parabolic fixed point.
In the case of two fixed points, the dynamics near each fixed
point is very simple: one is attractive and the other one is repulsive,
so they are both locally linearisable. On the other hand, when
we have a periodic orbit, the dynamics can be very complicated
when the multiplier lies on the unit circle. 

For a generic unfolding $f_\ep$ of an antiholomorphic parabolic germ $f_0$ we are able to identify a \emph{canonical parameter}, which is a real analytic invariant. Then an equivalence between two generic unfoldings will preserve the canonical parameter. A \emph{weak modulus of classification} is given by an unfolding of the modulus of $f_0$.  However, when we only work with real values of the parameters, we are not able to prove that the conjugacy between two unfoldings depends real analytically on the parameter: this is why we only speak of \emph{weak} modulus of classification. To remedy, we need to extend $f_\ep$ antiholomorphically to complex values of the parameter and to introduce a \emph{strong modulus of classification}. Then $g_\ep=f_\epbar\circ f_\ep$ depends holomorphically on $\ep$ and we are able to prove the existence of a conjugacy depending real analytically on the parameter. This is done using the tools developed for the holomorphic case.

After a classification problem is solved the natural question to ask is the \emph{modulus space}. 
We have been able to give the \emph{strong modulus set} (without a topology) for generic unfoldings  of an antiholomorphic parabolic germ of codimension 1, i.e.~necessary and sufficient conditions for a modulus to be realized. The proof uses a detour. Indeed the problem is solved for unfoldings $g_\ep$ of a parabolic germ (see \cite{realisation}). 
Hence it suffices to solve the problem of the existence of an antiholomorphic unfolding $f_\ep$ depending antiholomorphically on $\ep$ such that $g_\ep=f_\epbar\circ f_\ep$ (which we call an \emph{antiholomorphic square root}), and this problem is easy to solve.

The idea of the necessary condition for the realisation is the following: the description of the modulus is done in the parameter $\hatep$ belonging to the universal covering punctured at $0$ of the complexified canonical parameter: we take $\arg\hatep \in (-\pi+\delta, 3\pi -\delta)$ for some $\delta\in (0,\frac{\pi}2)$. Hence, when $\arg\ep\in (-\pi+\delta, \pi-\delta)$, we have two different unfoldings of the modulus at $\ep=0$. An obvious necessary condition is that these two different unfoldings describe the same dynamics. This condition is called the \emph{compatibility condition}, and it turns out to also be sufficient.
Hence to obtain the realisation in the antiholomorphic case it suffices to combine the compatibility condition and to solve the problem of finding  a necessary and sufficient condition for the extraction of an antiholomophic square root $f_\ep$ of a holomorphic parabolic unfolding $g_\ep$. By this we mean a germ depending antiholomorphically on the complex parameter $\ep$ and satisfying $g_\ep=f_\epbar\circ f_\ep$.

The study of the unfoldings of $f_0$ can help shed some
light on \emph{why} there are obstructions to certain
simple geometric behaviours. For instance, it is proved in~\cite{GR} that under
a condition of infinite codimension, $f_0$ will have an
invariant real analytic curve. What is the obstruction?
For a generic unfolding $f_\ep$ of $f_0$,
there will be some values of the parameter for which $f_\ep$
has two simple fixed points. The linearisation near each
fixed point has one invariant analytic curve. There
is no reason for these local invariant curves to match globally and this is why the mismatch is the generic situation in the unfolding.
If the mismatch persists at the limit when $\ep\to 0$, then
we can expect that $f_0$ has no real analytic invariant curve.

The second example, already mentioned above of a rare geometric phenomenon is 
the existence of an antiholomorphic square root for a holomorphic parabolic germ $g_0$.
Such a germ is said
to have an \emph{antiholomorphic square root} if there exists an antiholomorphic
parabolic germ $f_0$ such that $f_0\circ f_0 = g_0$. The existence
of such a germ is a phenomenon of infinite codimension. Why? 
When we unfold $g_0$ by a parameter $\ep$, each fixed point of 
$g_\ep = f_\ep \circ f_\ep$ has a first return map that 
describes the dynamics locally. For a generic unfolding of
$g_0$, the first return maps are independent.
However the existence of $f_\ep$ forces the first return maps
to be conjugate for parameter values for which $f_\ep$
has a periodic orbit. Since we can have very complicated dynamics of very diverse types
for many values of the parameter, we see that this is a 
very strong condition, which explains the very high codimension. 

The paper is organized as follows. 

After a brief section of preliminaries on antiholomorphic parabolic germs, we define in Section~\ref{sec:pt para} generic 
unfoldings $f_\ep$ of  antiholomorphic parabolic germs $f_0$ of codimension $1$. We also determine their canonical parameter $\ep$
and we give a \emph{prepared form} for generic
unfoldings.

In Section~\ref{sec:fatou}, we prove the existence of Fatou
coordinates for $f_\ep$. The Fatou coordinates are what
constitute the sectorial normalisation, i.e.~almost unique changes
of coordinates defined on portions of the domain
that conjugate $f_\ep$ to the \lq\lq normal form\rq\rq. 
%We introduce
%the time coordinate and the translation domains, where
%the Fatou coordinates will be defined. 

In Section~\ref{sec:orbits}, we describe  the space
of orbits of $f_\ep$ for the different values of $\ep$, using the transition functions between the Fatou coordinates.

In Section~\ref{sec:weak class}, we define the weak modulus of
classification and we prove a weak version of the Classification Theorem, where we miss the analytic dependence of the conjugacy on the parameter.

In Section~\ref{sec:strong class}, we give the strong version
of the Classification Theorem, which requires extending $f_\ep$ antiholomorphically to complex values of the parameter and defining a strong modulus of classification. 

Lastly, in Section~\ref{sec:app}, we apply the results to discuss the geometric interpretation of the modulus.  
We find a necessary and sufficient condition to extract an antiholomorphic square root
of a holomorphic unfolding of a parabolic point. We use it to give the modulus set.

\section{Preliminaries}
\subsection{Notation} For the whole paper,
we will use the following notation~:
\begin{itemize}
  \item $\sigma(z) = \zbar$ is the complex conjugation;
  \item $\tau(w) = {1 \over \overline w}$ is the antiholomorphic inversion;
  \item $T_C(Z) = Z + C$ is the translation by $C\in \Cp$;
  \item $L_c(w) = cw$ is the linear transformation with multiplier $c\in \Cp$;
%  \item $v^t$ is the time-$t$ of the vector field
  %  \begin{equation}\zpoint = v(z) = {z^{k+1} \over 1 + bz^k}.\end{equation}
\end{itemize}
\bigbreak

\subsection{Antiholomorphic Parabolic Fixed Points}
A function $f\colon U \to \Cp$ defined on a domain
$U\subseteq \Cp$ is antiholomorphic if
${\del f\over \del z} \equiv 0$ on $U$.
From this definition, together with the chain rule,
it follows that antiholomorphy is an intrinsic 
property of $f$ under holomorphic changes of variable.
Equivalently, $f\colon z\mapsto f(z)$ is antiholomorphic
if $f\circ\sigma\colon z\mapsto f(\zbar)$ is holomorphic, therefore
$f(z)$ expands in a power series in terms of $\zbar$.

Let $f\colon (\Cp,0) \to (\Cp,0)$ be a germ of antiholomorphic
diffeomorphism (in the $\zeta$-variable) that fixes the origin. Recall that $0$ is a
\emph{parabolic fixed point} if it is an isolated fixed point
and if $\left|{\del f\over \del \overline{\zeta}}(0)\right| = 1$. It is proved
in~\cite{GR} that there exists a polynomial change of coordinate
$z = p(\zeta)$ such that $f$ takes the form
\begin{equation}\label{eq:forme prenormale k}
  \ftilde(z) = \zbar + {1\over 2}\zbar^{k+1} 
    + \left({k+1 \over 8} - {b\over 2}\right) \zbar^{2k+1}
    + o(\zbar^{2k+1}),
\end{equation}
for some numbers $k\in \N^\ast$, $b\in\R$, respectively
called the \emph{codimension} and the \emph{formal invariant} 
of $f$. The codimension is linked to the multiplicity of the
fixed point: a fixed point of codimension~$k$ has multiplicity
$k+1$. We concentrate on the codimension~1 case, so
we assume that $k=1$ for the rest of the paper. We further
suppose that $f$ is always in a coordinate such that 
\begin{equation}\label{eq:prenormale 1}
  f(z) = \zbar + {1\over 2} \zbar^2 
    + \left({1 \over 4} - {b\over 2}\right) \zbar^3
    + o(\zbar^4).
\end{equation}
The formal invariant allows us to define a formal normal
form, namely $\sigma \circ  \vt$, the composition of the time-${1\over 2}$ map of
the vector field
\begin{equation}\label{eq:champ k}
  \dot z = v(t) = {z^{2} \over 1 + bz}
\end{equation}
with the complex conjugation. In other words, 
there exists a formal change of coordinate $h$ such that 
$h\circ f\circ h\inv = \sigma\circ \vt$.

The local dynamics of $f$ at its fixed point is also
described in~\cite{GR}. Let us briefly review the important
features. First, recall that the composite $f \circ f$ is a germ 
of holomorphic diffeomorphism with a holomorphic parabolic 
fixed point at the origin. The dynamics of $f\circ f$ is
embedded in the dynamics of $f$; this has several consequences,
for instance the codimension and the formal invariant of
$f\circ f$ are the same as the ones of $f$. It is well
known that the dynamics of $f\circ f$ is described using
the so-called \emph{Écalle horn maps}. 
In codimension~$1$, the space of orbits  of $f\circ f$
is a quotient of the
disjoint union of $2$ spheres by the identification coming from a pair of  germs $(\psi^0,\psi^\infty)$ 
of normalized diffeomorphisms, where $\psi^0$ sends $0$ to $0$ and $\psi^\infty$ sends $\infty$ to $\infty$, as in Figure~\ref{fig:esp orb g}, and by further identification of $0$ and $\infty$.
This orbit space depends only on the equivalence class $[\psi^0,\psi^\infty]$,
where the equivalence relation is given by
\begin{equation}\label{eq:relation g}
    (\psi^0,\psi^\infty)
      \sim_\Cp({\psi^0}',{\psi^\infty}')
     \Leftrightarrow \exists C\in \Cp,\ \psi^{0,\infty} 
      = L_C \circ {\psi^{0,\infty}}' \circ L_{-C}.
\end{equation}

\begin{figure}[htbp]
  \centering
  \includegraphics{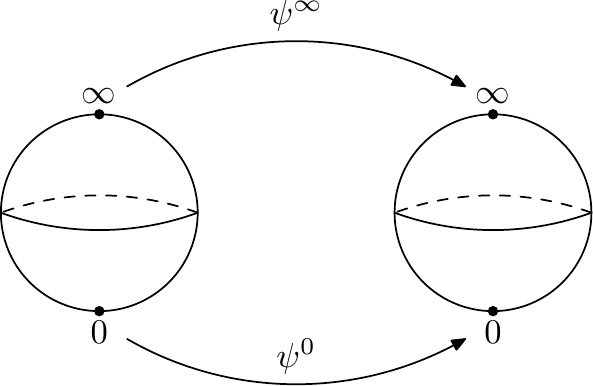}
  \caption{Space of orbits of $f\circ f$ in
  codimension~1.}
  \label{fig:esp orb g}
\end{figure}

In the antiholomorphic case, the fact that the dynamics of
$f\circ f$ is embedded in the dynamics of $f$ allows us to
define the involution $L_{-1}\circ \tau \colon w\mapsto -{1\over \wbar}$
on the space of orbits of $f\circ f$. We then obtain the space of
orbits of $f$ by quotienting the space of orbits of
$f\circ f$ by $L_{-1}\circ \tau$.
We are left with two projective spaces and a class
of germs of diffeomorphisms $[\psi]$ as in Figure~\ref{fig:esp orb f}.
The equivalence class is given by
\begin{equation}\label{eq:relation f}
  \begin{aligned}
    \psi
      \sim_\R \psi'
    \Leftrightarrow \exists R\in \R,\ \psi = L_R \circ \psi' \circ L_{-R}.
  \end{aligned}
\end{equation}
Moreover, if $\psi$ is a representative of $[\psi]$, then 
$(\psi^0, \psi^\infty)$ defined by
$$
  \left\{
  \begin{aligned}
    &\psi^\infty := \psi,\\[2\jot]
    &\psi^0 := L_{-1} \circ \tau \circ \psi^\infty \circ \tau \circ L_{-1},
  \end{aligned}
  \right.
$$
is a representative $[\psi^0,\psi^\infty]$  (the horn maps of $f\circ f$).
\begin{figure}[htbp]
  \centering
  \includegraphics{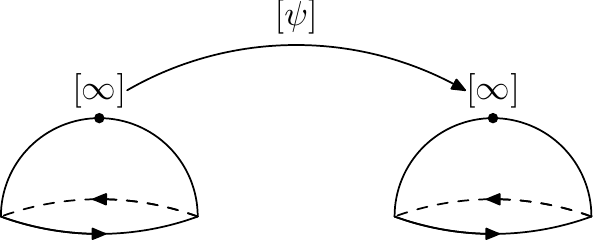}
  \caption{Space of orbits of $f$ in codimension~1.}
  \label{fig:esp orb f}
\end{figure}

\section{Generic Unfoldings of Antiholomorphic Parabolic Germs}\label{sec:pt para}

\subsection{Real parameters}
Note that it does not make sense in the context of iterations of antiholomorphic functions to speak of analytic dependence on complex parameters since this notion is not invariant under composition. What does make sense however is to speak of real analytic dependence on real parameters. And, indeed, 
a fixed point $z=0$ of $f$ is multiple as soon as $|f'(0)|=1$, which is generically a condition of real codimension $1$. Hence we should expect a real unfolding parameter to be transversal to this condition.

Since the \lq\lq natural\rq\rq\ parameters are real, we will use
mix anti-analyticity, which we introduce in the following definition.
Note that the even iterates will be mix analytic.

\begin{deff} Let $U\subseteq \R^n$ and $V\subset \Cp^m$
  be domains and $f\colon U\times V\to \Cp$ be a function.
  We say $f$ is \emph{mix analytic} (resp.~\emph{mix anti-analytic})
  if for every $(t^\ast, z^\ast) \in U\times V$, there exist an $n$-dimensional
  rectangle $R = R^n(t^\ast, r) \subseteq U$ and a polydisc
  $P = D^m(z^\ast,\rho) \subseteq V$, where $r\in\R_{>0}^n$, $\rho\in\R_{>0}^m$,
  such that $f$ has a convergent power series expansion in 
  $R\times P$ of the form
  $$
    f(t,z) = \sum_{\alpha \in \N^n} \sum_{\beta \in \N^m}
      \alpha_{\alpha,\beta} (t - t^\ast)^\alpha (z - z^\ast)^\beta,
  $$
  or, respectively,  of the form
  $$
    f(t,z) = \sum_{\alpha \in \N^n} \sum_{\beta \in \N^m}
      \alpha_{\alpha,\beta} (t - t^\ast)^\alpha (\zbar - \zbar^\ast)^\beta.
  $$
\end{deff}

\begin{lem} Let $U\subseteq \R^n$, $V \subseteq \Cp$ and
  $W\subseteq \Cp$ be domains. Let $F\colon U\times V\to\Cp$
  and $G\colon U\times W\to \Cp$ be functions. Whenever the composition
  is possible,
  \begin{enumerate}
    \item if $F$ and $G$ are mix analytic, then so is $F \circ (id\times G)$;
    \item if $F$ is mix analytic and $G$ is mix anti-analytic, then
      $F \circ (id\times G)$ and $G\circ (id\times F)$ are mix anti-analytic;
    \item if $F$ and $G$ are mix anti-analytic, then $F\circ (id\times G)$
      is mix analytic.
  \end{enumerate}
\end{lem}
\begin{proof} The proof is a simple computation of derivatives. Indeed,
  we can complexify $t\in U$, show that the composition is holomorphic
  for 1 and 3 or holomorphic in $t$ and antiholomorphic in $z$ for 2, 
and  then restrict $t$ to the reals.
\end{proof}

\begin{Wtheo}[Real analytic version]\label{theo:weierstrass}
  Let $g\colon (\R^N,0)\times (\Cp,0) \to (\Cp,0)$; $(t,z)\mapsto g(t,z)$
  be a germ of mix analytic function. If
  $$
    {\del^k g\over \del z^k}(0,0) = 0
    \qquad\hbox{ and }\qquad
    {\del^n g\over \del z^n}(0,0)\not=0 
    \eqno (k=1,\ldots,n-1),
  $$
  then there exists a polynomial $P_t(z) = z^n + a_{n-1}(t)z^{n-1} + \cdots + a_0(t)$,
  where $a_j$ is real analytic, $a_j(0)=0$, and a germ of mix analytic function $h\colon (\R^N,0)\times (\Cp,0)\to \Cp$  
   with $h(0,0)\not=0$ such that
  $$
    g(t,z) = P_t(z) h(t,z).
  $$
  Moreover, if $g(t,\zbar) = \ol{g(t,z)}$, then 
  $a_j$ is real valued.
\end{Wtheo}
\begin{proof}
  The proof is identical to the proof of \cite{Range}. The fact that
  the $a_j$'s are real analytic follows from the fact that
  $g$ is mix analytic and the Cauchy's formula in the $z$
  variable.

  For the second part, we have $g(t,z) = P_t(z) h(t,z)$ as in the
  statement of the theorem. If $\varphi(t)$ is a zero of $P_t$,
  then $\ol{\varphi(t)}$ is also a zero with the same multiplicity,
  since ${\del^j \over \del \zbar^j} g(t,\zbar) = \ol{{\del^j \over \del z^j} g(t,z)}$.
  If $\alpha_1(t),\ol{\alpha_1(t)},\ldots,\alpha_\ell(t),\ol{\alpha_\ell(t)}$
  are the complex roots of $P_t$,
  and $\beta_j(t)$, $j=1,\ldots,\ell'$ are its real roots, then
  we have
  \begin{align*}
    P_t(z) &= (z - \alpha_1)^{q_1}(z - \ol{\alpha_1})^{q_1}
      \cdots (z - \alpha_\ell)^{q_\ell}(z - \ol{\alpha_\ell})^{q_\ell}
      (z-\beta_1)^{m_1}\cdots (z- \beta_{\ell'})^{m_{\ell'}}.
  \end{align*}
  For $x$ real, the polynomial $(x - \alpha_j)(x - \ol{\alpha_j})$ 
  is real valued, so $x\mapsto P_t(x)$ is real valued. It follows
  that its coefficients $a_j(t)$ are real for every $t$.
\end{proof}

\subsection{Generic Unfoldings}

We are now ready to define unfoldings of an antiholomorphic 
parabolic fixed point and the notion of equivalence of such 
unfoldings. {As discussed above it is natural to work with one real parameter. }

\begin{deff}
  \begin{enumerate}
    \item Let $f_0\colon ({\Cp},0)\to (\Cp,0)$ be a germ of antiholomorphic
      diffeomorphism with a parabolic fixed point (antiholomorphic
      parabolic germ for short) of codimension~1. An \emph{unfolding}
      of $f_0$ is a germ of mix anti-analytic diffeomorphism
      $f\colon ({\R},0)\times (\Cp,0) \to (\Cp,0)$, {$(\ep, z)\mapsto f(\eps,z)=f_\ep(z)$. }
    \item { We can of course suppose that $f_0$ is in the form \eqref{eq:prenormale 1}. 
    Then the unfolding has the form 
    \begin{equation}\label{unfolding}f_\ep(z)= f_0(z) + \sum_{j\geq0} a_j(\eps)\zbar^j,\end{equation} with $a_j(0)=0$.} We say that the unfolding is \emph{generic} if
    \begin{equation}
        \left.{\del \Re(a_0)\over \del \ep}\right|_{(\ep,z)=(0,0)} \not= 0.
    \end{equation}
  \end{enumerate}
\end{deff}

The equivalence of two families will be defined in terms of
mix analyticity.

\begin{deff}\label{def:eq forte}
  Let $f_{1,\eta}$ and $f_{2,\ep}$ be two generic unfoldings
  of antiholomorphic parabolic germs of codimension $1$. We say they are
  \emph{equivalent} if there exists an open interval
  $I\ni 0$, a disc $D(0,r)$, $r>0$, and a mix analytic
  diffeomorphism $H\colon I\times D(0,r) \to \R\times \Cp$
  such that
  \begin{enumerate}
    \item $H(0,0) = (0,0)$;
    \item $H(\eta,z) = (\beta(\eta), h_\eta(z))$ with 
      $\beta$ real analytic and $h_\eta$ mix analytic;
    \item $f_{2,\beta(\eta)} = h_\eta \circ f_{1,\eta} \circ h_{\eta}\inv$.
  \end{enumerate}
\end{deff}

The main goal of the
paper is to describe the equivalence classes.

\subsection{Canonical Parameter}
In the holomorphic case, it is proved in~\cite{germeDeploie} that a
generic unfolding of a parabolic fixed point has exactly
one canonical complex parameter.  In the same way, a generic unfolding in 
the antiholomorphic case will also have exactly one real
canonical parameter.

\begin{lem}\label{lem:pts fixes}
  Let $f_0$ be a antiholomorphic parabolic germ of codimension~1
  and let $f_\ep$ be a generic unfolding of $f_0$. Then there exists
 a change of coordinate and parameter $(z,\ep)\mapsto (Z,\eta)$ such that the fixed points are located at $Z^2=\eta$ whenever they exist.
\end{lem}
\begin{proof}
  Suppose that $f_\ep$ has the form~\eqref{unfolding}. 
  Let $z=x+iy$ and set $F(x,y,\ep) = f_\ep(z) - z$ and $F_1 = \Re F$,
  $F_2 = \Im F$. We find
  \begin{align}\begin{split}
    F_1(x,y,\ep) &= \Re a_0(\ep) + O(\ep)O(|x,y|) 
      + \big(1 + O(\ep)\big){x^2 - y^2 \over 2}
      + O(|x,y|^3),\\[3\jot]
    F_2(x,y,\ep) &= \Im a_0(\ep) - 2y +O(\ep)O(|x,y|) 
      - \big(1 + O(\ep)\big) xy 
      + O(|x,y|^3).
 \end{split} \label{eq:F}\end{align}
 Since 
  ${\del F_2 \over \del y}(0) = -2$, by the Implicit Function Theorem,
  there exists a real analytic function $m$ such that
  $F_2 = 0$ if and only if $y = m(x,\ep) = O(|x,\ep|^2)$.
  
  We make the change of variable $z= z_1+im(z_1,\ep)$, which sends the real axis in $z_1$-space to $y=m(x,\ep)$ in $z$-space. Let $f_{1,\ep}$ be the expression of $f_\eps$ in the new variable $z_1$. 
Let us now consider the corresponding equations~\eqref{eq:F} in the new variable $z_1$. For $y_1=0$, the first equation has the form 
$\tilde F_1(x_1,\ep) = x_1^2 (1+O(\ep)+O(x_1))+ O(\ep)x_1 + O(\ep)$, and we have
  ${\del \tilde F_1\over \del x_1}(0) = 0$ and 
  ${\del^2 \tilde F_1 \over\del x_1^2}(0) = 1.$
  By the Weierstrass Preparation Theorem~\ref{theo:weierstrass},
  there exists a polynomial $P_\ep(x_1) = x_1^2 + \alpha_1(\ep) x_1 + \alpha_0(\ep)$, with $\alpha_0(\ep)$ and $\alpha_1(\ep)$ real and $\alpha_0'(0)\neq0$,
  such that  $\tilde F_1(x_1,\eps)=0 $, if and only if $P_\ep(x_1)=0$.
  
  A translation $Z= z_1 +\frac12 \alpha_1(\ep)$ brings the fixed points to $Z^2=\eta$, where $\eta= \frac14\alpha_1^2(\eps)-\alpha_0(\eps)$. 
\end{proof}

\subsubsection{Normal Form}
We consider the vector field
\begin{equation}
  \zpoint = v_\ep(z) = {z^2 - \ep \over 1 + b(\ep) z}\label{vector_field}
\end{equation}
and the time-$t$ map $v^t_\ep$, where $b\colon(\R,0)\to\R$ is
a germ of real analytic function. 
Then $\sigma\circ \vt_\ep$ is a generic unfolding of a antiholomorphic parabolic germ of codimension 1.  It will be the natural \lq\lq model\rq\rq\ (normal form) to which we will compare any generic unfolding $f_\eta$: unique new parameter $\ep$ and function $b(\ep)$ will be found  so that the multipliers of $f_\ep^{\circ 2}$ at its fixed points be the same as those of  $(\sigma\circ \vt_\ep)^{\circ 2}$ at its fixed points.

Since $\sigma \circ v_\ep = v_{\ep} \circ \sigma$, it follows that
$\sigma \circ \vt_\ep = \vt_\ep \circ \sigma$. In particular,
we have $(\sigma \circ \vt_\ep)^{\circ 2} = \vun_\ep$, which
correspond to the normal form in the holomorphic case for $\ep$ real.

The multipliers of $v_\ep$ are given by
$$
  \mu_\pm = \pm {2\sqrt{\ep} \over 1 \pm b(\ep)\sqrt{\ep}}
$$
and the multipliers of $\vun_\ep$ are
$\lambda_\pm = \exp(\mu_\pm)$. It follows that
\begin{equation}
  \ep := \left({1 \over \log(\lambda_+)}
    - {1 \over \log(\lambda_-)}\right)^{-2}\kern-4mm,
  \qquad\qquad
  b(\ep) := {1 \over \log(\lambda_+)}
    + {1 \over \log(\lambda_-)}.
\label{b-ep}\end{equation}
Since the multipliers are preserved by analytic changes of
coordinate, we see that  $b$ and $\ep$ are invariant. In particular, that gives us the hint on how to find the canonical parameter $\ep$ of an arbitrary antiholomorphic parabolic germ.

\subsubsection{Canonical Parameter and Prepared Form}
The formula \eqref{b-ep} allows to define the \emph{canonical parameter} and the \emph{formal invariant}  of the unfolding.
These are the formal part of the modulus of classification
that we will define in Section~\ref{sec:weak class}. The analytic part of the modulus will consist in a measure of the obstruction to a conjugacy of an unfolding to its model.

\begin{theo}[Canonical Parameter]\label{theo:ep canonique}
  Let $\ftilde_0\colon (\Cp,0)\to (\Cp,0)$ be an antiholomorphic
  parabolic germ of codimension~1 with formal invariant $b_0\in\R$.
  Let $\ftilde_\eta$ be a generic unfolding depending 
  on the real parameter  $\eta$ with fixed points at $z^2=\eta$.
  Let $\gtilde_\eta = \ftilde_\eta \circ \ftilde_\eta$. We set
  \begin{gather}\label{eq:def ep}
    \ep := \left({1 \over \log(\tilde g_\eta'(\sqrt\eta))}
      - {1 \over \log(\tilde g_\eta'(-\sqrt\eta))}\right)^{-2}\kern-4mm,
      \\[2\jot]\label{eq:def b}
    b := {1 \over \log(\tilde g_\eta'(\sqrt\eta))}
      + {1 \over \log(\tilde g_\eta'(-\sqrt\eta))},
  \end{gather}
  where the logarithm is the principal branch. Then we have that
  $\ep$ and $b$
  \begin{enumerate}
    \item are real-valued;
    \item are invariant under changes of coordinate;
    \item can be continued into real analytic functions,
      in particular $\ep(0) = 0$ and $b(0) = b_0$.
  \end{enumerate}
  The parameter $\eps$ is called the \emph{canonical parameter} 
  and the function $b(\eps)$ the \emph{formal invariant} of the family.
\end{theo}

We postpone the proof to introduce the \emph{prepared form}. When comparing 
unfoldings, it will be useful to compare them when they are
in their prepared form, since the prepared form depends on the
canonical parameter. A prepared form for the holomorphic
case was presented in~\cite{germeDeploie}.

\begin{theo}[Prepared Form]\label{theo:forme prep}
  Under the the hypotheses of Theorem~\ref{theo:ep canonique},
  there exists a mix analytic diffeomorphism 
  $H\colon (\eta, z) \mapsto \big(\ep, m_\eta(z)\big) = (\ep,z_1)$
   that maps the family ${\{\ftilde_\eta\}}_\eta$ to
  a family in the \emph{prepared form}
\begin{equation}\label{eq:forme preparee}
    f_\ep(z_1) = \overline{z}_1 + (\overline{z}_1^2 - \ep)
      [B_0(\ep) + B_1(\ep)\overline{z}_1 
      + (\overline{z}_1^2 - \ep)Q(\ep,\overline{z}_1)],
  \end{equation}
  that satisfies 
  \begin{enumerate}
    \item $B_0(0) = {1\over 2}$ and $B_0$ and $B_1$ are real-valued;
    \item $\tau_\pm:= {\del f_\ep\over\del \overline{z}_1}(\pm\sqrt{\ep})\begin{cases}
        \in\R, & \hbox{if $\ep > 0$;}\\
      = \ol{\tau_{\mp}},& \hbox{if $\ep < 0$;}\\
      = 1,& \hbox{if $\ep = 0$;}\end{cases}$
    \item $\lambda_\pm:= {\del g_\ep\over\del z_1}(\pm\sqrt{\ep}) = \ol{\tau_\pm^2}$, where $g_\ep := f_\ep\circ f_\ep$.  \end{enumerate}
\end{theo}

\subsubsection{Proof of Theorems~\protect\ref{theo:ep canonique} and~\protect\ref{theo:forme prep}}
For the proof of both theorems, by Lemma~\ref{lem:pts fixes}
and the Weierstrass Division Theorem, we can suppose
$\tilde f_\eta$ has the form 
\begin{equation}
  \tilde f_{\eta}(z) = \zbar + (\zbar^2 - \eta)\big(C_0(\eta) + C_1(\eta)\zbar 
    + (\zbar^2 - \eta) R(\eta,\zbar)\big),
\label{f_tilde}\end{equation}
with $C_0(0) = {1\over 2}$, $C_1(0) = {1\over 4} - {b(0)\over 2}$
and $R$ a mix analytic function. 
With the same
reasoning applied to $\tilde g_{\eta} := \tilde f_{\eta}\circ \tilde f_{\eta}$, we find
$$
  \tilde g_{\eta}(z) = z + (z^2 - \eta)\big(D_0(\eta) + D_1(\eta)z
    + (z^2 -\eta) Q(\eta,z)\big),
$$
with $D_0(0) = 1$, $D_1(0) = 1 - b(0)$
and $Q$ a mix analytic function. 

Lastly, we set
$$
  \tilde \tau_\pm := \tilde f_\eta'(\pm\sqrt\eta) = 1 \pm 2\ol{\sqrt\eta}\big(C_0 \pm C_1 \ol{\sqrt\eta}\big).
  %\qquad\qquad\quad \tilde\lambda_\pm := \tilde g'(\pm\sqrt\eta).
$$

\begin{proof}[Proof of Theorem~\ref{theo:ep canonique}]
  We complexify $\eta$. The fact that $\ep$ and
  $b$ are invariant is because the multipliers of
  $\tilde g_\eta$ are invariant. They are holomorphic 
  for $\eta\not=0$ since they are invariant under
  the permutation $\sqrt\eta \mapsto -\sqrt\eta$.
  Then we see they are bounded around $\eta=0$ using
  $\lambda_\pm = 1\pm D_0\sqrt\eta + O(\eta)$.
  The details are found in~\cite{germeDeploie}.
  
  Lastly, $\ep$ and $b$ are real-valued for $\eta$ real
  since the multipliers satisfy
  $$
    \lambda_\pm = \begin{cases}
      |\tau_\pm|^2, 
        & \text{if $\eta\geq0$;}\\
      \tau_\mp\ol{\tau_\pm}, 
        & \text{if $\eta < 0$;}
    \end{cases}
%    \tilde g_\eta'(\pm\sqrt\eta) 
%      = {\del \tilde f_\eta\over \del \zbar}(\pm\ol{\sqrt\eta})
%    \ol{{\del \tilde f_\eta\over\del\zbar}(\pm\sqrt\eta)},
  $$
  which follows from the chain rule.
\end{proof}

\begin{proof}[Proof of Theorem~\ref{theo:forme prep}] 
  Let $z_1=w_\eta(z) = z + (z^2 - \eta)\big(A_0(\eta) + A_1(\eta)z\big)$
  be a change of coordinate, where 
  $$
    A_0 = {\ol{\sqrt{\tilde\tau_+}} - \ol{\sqrt{\tilde\tau_-}}
      \over 4\sqrt\eta},\qquad\qquad
    A_1 = {\ol{\sqrt{\tilde\tau_+}} + \ol{\sqrt{\tilde\tau_-}} - 2
      \over 4\eta}.
  $$
  Let us complexify $\eta$.
  For $\eta\not=0$, $A_0$ and $A_1$ are analytic since they are
  invariant under the permutation $\sqrt\eta \mapsto -\sqrt\eta$.
  Also we have that 
  $\sqrt{\tilde\tau_\pm} = 1\pm\sqrt\eta + O(\eta)$, so $A_0$
  and $A_1$ are bounded when $\eta\to 0$. It follows that
  $A_0$ and $A_1$ are analytic. By direct computation, we see that
  $w_\eta'(\pm\sqrt\eta) = \ol{\sqrt{\tilde \tau_\pm}}$. Hence
  the derivatives of $f^\dag_\eta = w_\eta \circ \tilde f_\eta \circ w_\eta\inv$
  at $z_1 = \pm\sqrt\eta$ satisfies point 2. Also, point 1 is
  a consequence of point 2, and point 3 follows from
  the chain rule and point 2.

  We do a last change of coordinate and parameter, that will preserve
  points 1 to 3, so that the new family depends on its canonical parameter
  $\ep$. Let $\joli L(\eta,z_1) = \big(\ep(\eta), c(\eta)z_1\big) = \big(\ep,L_\eta(z_1)\big)$, where 
  $\ep\colon \eta\to\ep(\eta)$ is given by~\eqref{eq:def ep} and 
  $c(\eta) := \sqrt{\ep(\eta)/\eta}$. Since $\ep(\eta) = \eta\big(1 + O(\eta)\big)$,
  we see that the function $c$ is real analytic and real valued, so 
  $f_\ep = L_\eta\circ f^\dag_\eta\circ L_\eta\inv$ still satisfies
  points 1 to 3. Calling $z_2=c(\eta)z_1$, then we have
  $$
    f_\ep(z_2) = \zbar_2 + (\zbar_2^2 - c^2\eta)\big(B_0(\ep) + B_1(\ep)\zbar_2
      + (\zbar_2^2 - \ep) Q(\ep,\zbar_2)\big),
  $$
  for some real valued $B_0$ and $B_1$ and some remainder $Q$.
  Since $c^2\eta = \ep$, $f_\ep$ is in prepared form.
\end{proof}

\section{Fatou Coordinates}\label{sec:fatou}
Let $f_0$ be an antiholomorphic parabolic
germ of codimension~1. From now on,
we only consider prepared generic unfoldings 
depending on the canonical parameter $\ep$, that is
\begin{equation}
  f_\ep(z) = \zbar + (\zbar^2 - \ep)\big(B_0(\ep) + B_1(\ep)\zbar 
    + (\zbar^2 - \ep) Q(\ep,\zbar)\big),
\end{equation}
where $B_0$ and $B_1$ are real valued.

We will work with a representative of the germ,
also noted $f_\ep$,
defined on $(-r',r')\times D(0,r)$, where
$r',r >0$ may be as small as we want.

\subsection{Time coordinate}

We will work in the \emph{time coordinate} 
 of the vector field~\eqref{vector_field}. On a given simply connected domain in a disk $D(0,r)\setminus\{\pm\sqrt{\ep}\}$, it is defined up to a constant. Moreover it is ramified with period $2\pi i b(\eps)$ when turning around the two singular points as soon as $b(\ep)\neq0$: see Figure~\ref{fig:surface temps ep=0}.
 We will use two charts covering the two halves of $D(0,r)\setminus\{\pm\sqrt\ep\}$. 
 
 \begin{figure}[htbp]
  \centering
  \includegraphics{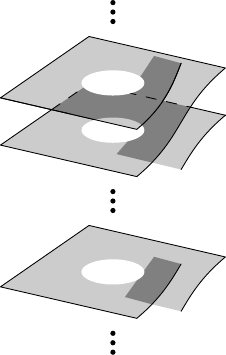}
  \caption{Riemann surface of the time coordinate close to the boundary of
  $D(0,r)$.}
  \label{fig:surface temps ep=0}
\end{figure}

\begin{prop}\label{prop:time}
 We define
  \begin{align}\begin{split}
      Z_\ep^+(z) &= \int_{r}^z {1 + \zeta b(\ep)\over \zeta^2 -\ep}\, \d\zeta,\\[2\jot]
        Z_\ep^-(z) &= \int_{-r}^z {1 + \zeta b(\ep)\over \zeta^2 -\ep}\, \d\zeta. 
  \end{split}\label{eq:coordonnee temps} \end{align}
The principal branches of $Z^\pm_\ep$ are defined by~\eqref{eq:coordonnee temps} 
 on  simply connected domains inside $\Cp\setminus\{-\sqrt\ep,\sqrt\ep\}$ containing $\pm r$.  
   
  Chosen in this way, these principal branches have the following properties:     
  \begin{enumerate}
    \item They satisfy
      \begin{equation}\label{eq:coordonnee temps_neg}
        Z_\ep^+(z) - Z_\ep^-(z) = \begin{cases}
          i\pi b(\ep), &\text{ if $\Im z > 0$  and $|z|=r$};\\[2\jot]
          -i\pi b(\ep),& \text{ if $\Im z < 0$ and $|z|=r$};
      \end{cases}
   \end{equation}
    \item They conjugate 
      $\vun$ on $T_1$;
    \item $Z_\ep^\pm(\zbar) = \ol{Z_\ep^\pm(z)}$.
  \end{enumerate}
 \end{prop}
\begin{proof}These properties are shown by direct computation.
Note that points 1 and 3 use the
fact that $\ol{b(\ep)} = b(\epbar)$, which follows from
Theorem~\ref{theo:ep canonique}.\end{proof}

If we look at $v_\ep^1$ on a disk
$D(0,r)$, then for $\ep \ll r$, the dynamics 
of $v_\ep^1$ is almost identical to the dynamics of $v_0^1$
near the boundary of $D(0,r)$. Therefore, the
time coordinate for $\ep$ small is similar to
the time coordinate for $\ep=0$ close to the
hole. See Figure~\ref{fig:surface temps ep=0}.

To obtain the rest of the surface corresponding to the
interior of $D(0,r)$, the path of integration in~\eqref{eq:coordonnee temps}
will have to turn around the ramification points
$\pm\sqrt\ep$. The surface of Figure~\ref{fig:surface temps ep=0}
will repeat itself with periods given by
\begin{equation}\label{eq:periode}
  \alpha_\ep^\pm := \int\limits_{{\gamma_\ep^\pm}} {1 + b(\ep)\zeta \over \zeta^2 - \ep}\d\zeta
    = \pm {i\pi\over \sqrt\ep} + i\pi b(\ep),
\end{equation}
where $\gamma_\ep^\pm$ is a small simple curve surrounding exactly the singular point $\pm \sqrt\ep$.
In particular, $T_{\alpha_\ep^\pm}$ are covering transformations of the
surface. 

Let $\Sigma$ be the lift of $\sigma$ on the time coordinate.
By point 2 of the Proposition~\ref{prop:time}, we see that $\Sigma$ is the
analytic continuation of the complex conjugation on the chart
of the principal branch obtained with the relation
$$
  \Sigma\circ T_{\alpha_\ep^\pm} = T_{\ol{\alpha_\ep^\pm}}\circ \Sigma.
$$

\subsection{Translation Domains}
We define charts in the time coordinate
determined naturally by the dynamics of $g_\ep = f_\ep\circ f_\ep$.
Let $G_\ep$ be the lift of $g_\ep$ in the
time coordinate. 

In~\cite{germeDeploie}, it is proved that 
$|G_\ep - T_1| \leq C\max\{r,r'\}$. Therefore,
$G_\ep$ is as close to $T_1$ as we want for
$r,r'$ small enough. If we take a vertical line $\ell$ or, for any $\beta\in (0,\frac{\pi}2)$, a slanted line $\ell$ making an angle in $(\beta, \pi - \beta)$ with the horizontal direction, then for $r,r'$ small enough, 
$\ell$ and $G_\ep(\ell)$ will not intersect.

\begin{rem} Occasionally we will suppose that $z=r$ belongs to the domain. This is legitimate by slightly restricting $r$. \end{rem}

\begin{deff}\label{def:trans dom} Let $\ell$ be a vertical line
with the distance between $\ell$ and
the holes at least $2$. We consider
the vertical strip $B_\ell$ lying between
$\ell$ and $G_\ep(\ell)$, including its
boundary. We define the \emph{translation domain} by
$$
  U_\ep = \{G_\ep^{\circ n}(Z) \mid Z\in B_\ell\}.
$$
The translation domain is called  a 
\emph{Glutsyuk} (resp.~\emph{Lavaurs}) translation
domain when $\ep>0$ (resp. $\ep<0$).
\end{deff}

For $\ep > 0$ (resp.~$\ep<0$), the holes are
aligned vertically (resp.~horizontally). The
strip $B_\ell$ is parallel (resp.~transversal)
to the line of the holes. 
See Figures~\ref{fig:glutsyuk} and~\ref{fig:lavaurs}.

\begin{figure}[htbp]
  \centering
    \subcaptionbox{Glutsyuk translation domain\label{fig:glutsyuk}}
      [.4\textwidth]{\includegraphics[scale=.5]{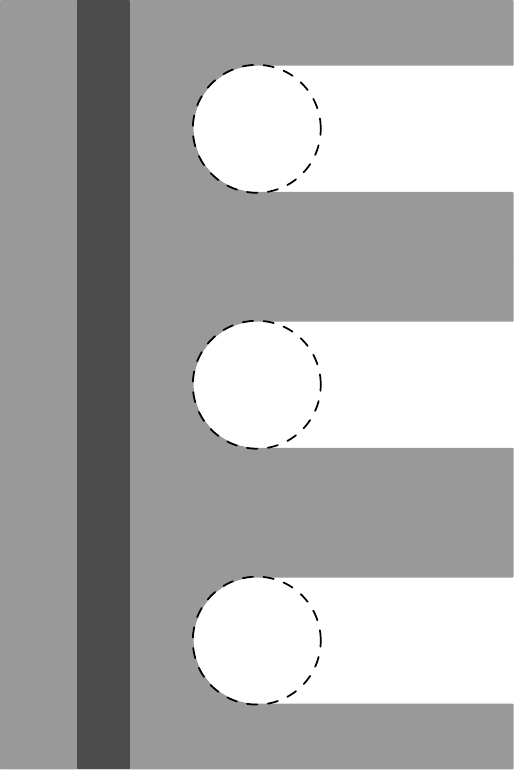}}
    \subcaptionbox{Lavaurs translation domain\label{fig:lavaurs}}
      [.4\textwidth]{\includegraphics[scale=.5]{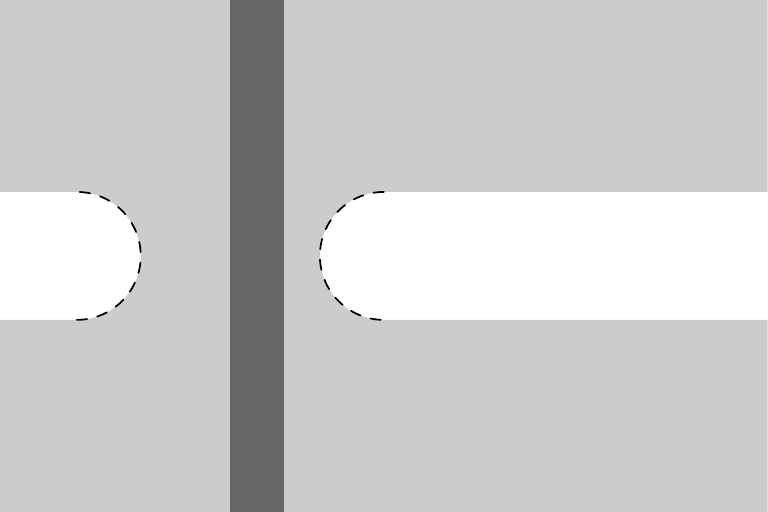}}
    \caption{Translation domains for $\ep > 0$ (left) and
      $\ep < 0$ (right).}
\end{figure}

\subsection{Existence of Fatou Coordinates}
The sectorial normalization of $f_\ep$ can 
easily be proved using the analogous theorem for
$f_\ep\circ f_\ep$. 

\begin{theo}\label{theo:coord fatou}
  Let $f_\ep$ be a generic unfolding in prepared form
  of an antiholomorphic parabolic germ of codimension~1.
  Let $F_\ep$ be the lift of $f_\ep$ on the time coordinate.
  \begin{enumerate}
    \item (Existence) For every $\ep>0$ (resp.~$\ep<0$),
      on every Glutsyuk (resp.~Lavaurs) translation domain
      $U_\ep^\pm$, there exists a diffeomorphism 
      $\Phi_\ep^\pm\colon U_\ep^\pm \to \Cp$ such that
      \begin{equation}\label{eq:Fatou}
        \Phi_\ep^\pm\circ F_\ep\circ (\Phi_\ep^\pm)\inv
          = \STt.
      \end{equation}
    \item (Uniqueness) If $\tilde \Phi_\ep^\pm$ is another
      diffeomorphism satisfying~\eqref{eq:Fatou}, then
      there exists a real constant $R_\pm(\ep)$ such that
      $$
        \tilde \Phi_\ep^\pm \circ (\Phi_\ep^\pm)\inv = T_{R_\pm(\ep)}.
      $$
    \item For $\ep>0$, the Fatou coordinates on $U_\ep^\pm$
      commute with $T_{\alpha_\ep^\pm}$, where $\alpha_\ep^\pm$
      is the period~\eqref{eq:periode}.
    \item (Dependence on parameter) Let
      $$
        Q^\pm = \bigcup_\ep \{\ep\}\times U_\ep^\pm.
      $$
      Then $Q^\pm$ is open in $\R\times \Cp$ and
      there exists a family of Fatou coordinates ${\{\Phi_\ep^\pm\}}_\ep$
      continuous on $Q^\pm$ and mix analytic for $\ep\not=0$.
     The family is uniquely determined by
      \begin{equation}\label{eq:determine Phi ep reel}
        \Phi_\ep^\pm(X_\ep) = C(\ep),
      \end{equation}
      where $X_\ep$ is a base point, $C$ is real-valued,
      and both $X_\ep$ and $C$ are real analytic in $\ep\neq0$ 
      with continuous limit at $\ep=0$.
  \end{enumerate}
\end{theo}
\begin{proof}
  1 and 4.~It is proved in~\cite{germeDeploie} that we can
  construct Fatou coordinates of $g_\ep$ on translation domains.
  A Fatou coordinate $\Phi_\ep^\pm\colon U_\ep^\pm\to \Cp$
  on the translation domain $U_\ep^\pm$ is a change
  of coordinate that rectifies $G_\ep$ to the normal
  form: $\Phi_\ep^\pm\circ G_\ep \circ (\Phi_\ep^\pm)\inv = T_1$.
  We can choose a family ${\{\Phi_\ep^\pm\}}_\ep$ that is 
  continuous in $(\ep,Z)$ and mix analytic for $\ep\not=0$,
  it suffices for instance to choose Fatou coordinates 
  with a fixed base point $Z_0$: $\Phi_\ep(Z_0)=0$.
  Note that on a Glutsyuk translation domain $U_\ep^\pm$ ($\ep > 0$),
  Fatou coordinates commute with $T_{\alpha_\ep^\pm}$.

  Let $P_\ep = \Phi^\pm_\ep\circ F_\ep\circ (\Phi^\pm_\ep)\inv$.
  We have $P_\ep\circ P_\ep = T_1$, so $P_\ep$ and $T_1$ commute.
  We prove that $\Sigma\circ P_\ep$ is a translation. 
  
  In the Fatou coordinate of the Glutsyuk translation domain
  ($\ep > 0)$,
  we can describe a fundamental domain of the orbits of
  $G_\ep$ on $U^\pm_\ep$ by quotienting the
  space by $T_1$ and $T_{\alpha_\ep^\pm}$. This results in a torus,
  where each point represents an orbit of $G_\ep$ on $U_\ep^\pm$.
  Since $\Sigma\circ P_\ep$ commutes with $T_{\alpha_\ep^\pm}$ and $T_1$, it
  induces a holomorphic diffeomorphism of this torus on itself. Such a
  mapping must be a translation, so it
  follows that $P_\ep$ is of the form $\Sigma\circ T_{C(\ep)}$,
  for some $C(\ep)$.
  
  Similarly, in the Fatou coordinate of the Lavaurs translation
  domain ($\ep < 0$), a fundamental domain of the orbits of $G_\ep$ on
  $U^\pm_\ep$ is obtained by the quotient
  of $T_1$, yielding the well-known \'Ecalle cylinder. The cylinder
  is isomorphic to $S^2 \setminus\{0,\infty\} = \Cp^\ast$. Each
  point of $\Cp^\ast$ corresponds to an orbit of $G_\ep$ on
  $U_\ep^\pm$, and we can fill the holes at $0$ and $\infty$ by 
  identifying them to the fixed points.
  Since $\Sigma\circ P_\ep$ commutes with $T_1$ and $f$ maps $\pm\sqrt\ep$
  to $\mp\sqrt\ep$, $\Sigma\circ P_\ep$ induces a holomorphic diffeomorphism
  of the sphere on itself, and such a mapping must be a
  linear map. It follows that $P_\ep$ is of
  the form $\Sigma \circ T_{C(\ep)}$ for some $C(\ep)$.
   
  In both cases, $P_\ep$ is of the form $\Sigma\circ T_{C(\ep)}$, for
  some $C(\ep) = {1\over 2} + iy(\ep)$, $y\in\R$, where $y$ is real analytic in $\ep\neq0$ with continuous limit at $\ep=0$. It follows that
  $T_{iy(\ep)/2}\circ \Phi^\pm_\ep$ is a Fatou coordinate of $f_\ep$.
  That~\eqref{eq:determine Phi ep reel} determines the
  family ${\{\Phi_\ep^\pm\}}_\ep$ follows from point 2 below.
 
  2.~Recall that the Fatou coordinates of $g_\ep$ are unique up
  to translation. To preserve~\eqref{eq:Fatou}, the constant of
  translation must be real.

  3.~This is true for Fatou coordinates of $g_\ep$ (see~\cite{germeDeploie}). 
  
\end{proof}

\begin{cor} Each strip $B_\ell$ is conformally equivalent to a sphere minus two points. 
\end{cor}

\section{Space of orbits}\label{sec:orbits}
% TODO Écrire intro de la section quand la section
% sera un peu plus complète.

\subsection{Transition Functions}

Equation~\eqref{eq:coordonnee temps_neg} from Proposition~\ref{prop:time} 
allows us to define a transition function 
$T_{-i\pi b}= Z_\ep^- \circ (Z_\ep^+)\inv$ on the 
connected component above the fundamental hole in
$U_\ep^+$ to the corresponding component in $U_\ep^-$,
see Figure~\ref{fig:Tipib}.
Since $T_{-i\pi b}$ commutes with $T_1$, it sends orbits of the holomorphic normal form $v_\ep^1$ on orbits of $v_\ep^1$.
\begin{figure}[htbp]
    \subcaptionbox{Transition functions on a Glutsyuk domain
      \label{fig:transition ep>0}}
      {\includegraphics[scale=1]{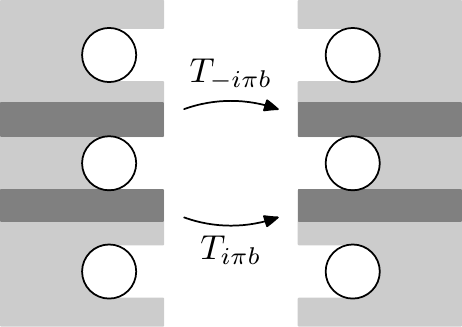}}\hfill
    \subcaptionbox{Transition functions on a Lavaurs domain
      \label{fig:transition ep<0}}
      {\includegraphics[scale=1]{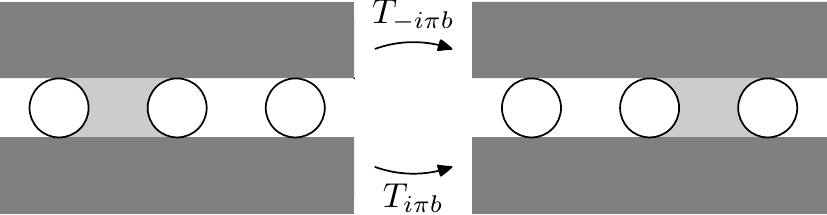}}
    \caption{Transition functions between time coordinates
      for $\ep > 0$ (left) and $\ep < 0$ (right).}
      \label{fig:Tipib}
\end{figure}

We define the \emph{transition functions} of $f_\ep$
that identify orbits represented on both $U_\ep^+$ and $U_\ep^-$.

\begin{deff}\label{def:psi}
  Let $(\Phi_\ep^+,\Phi_\ep^-)$ be a pair of Fatou
  coordinates of $f_\ep$ on $U_\ep^+$ and $U_\ep^-$
  respectively. Let $V_\ep^\pm = \Phi_\ep^\pm(U_\ep^\pm)$.
  We define the \emph{transition function} of $f_\ep$
  on the connected component above the fundamental hole 
  where the composition is defined by
 \begin{equation}\label{def_psi}
    \Psi_\ep = \Phi_\ep^- \circ T_{-i\pi b} \circ (\Phi_\ep^+)\inv.
  \end{equation}
 
\begin{notation} We will note it $\Psi^\infty_\ep$ for $\ep\leq 0$, and
  $\Psi^G_\ep$ for $\ep > 0$. \bigskip

  For $\ep \leq 0$, we define a transition function
  below the fundamental hole by
  $$
    \Psi_\ep^0 =\Phi_\ep^- \circ T_{i\pi b} \circ (\Phi_\ep^+)\inv. 
  $$
\end{notation}
\end{deff}

The most important property of the transition functions is that they commute
with $T_1$, so that they map orbits of $g_\ep$ on
orbits of $g_\ep$. 

Since they are defined using the Fatou coordinates of
$f_\ep$, they are also compatible with the orbits of 
$f_\ep$. For $\ep \leq 0$, we have
\begin{equation}\label{eq:PsiInf commute}
  \Psi_\ep^\infty \circ \STt = \STt \circ \Psi_\ep^0,
\end{equation}
and for $\ep > 0$, if we define $\Psi_\ep^U = \Psi_\ep^G$ and 
$\Psi_\ep^L = T_{\alpha_\ep^-}\circ \Psi_\ep^G\circ T_{\alpha_\ep^+}$ then
\begin{equation}\label{eq:PsiG commute}
  \Psi_\ep^U \circ \STt = \STt \circ \Psi_\ep^L.\end{equation} 
(See Figure~\ref{fig:strips}.)  

\begin{figure}[htbp]
    \centering
    \subcaptionbox{Horizontal strips and transition functions
      \label{fig:strips}}
      {\includegraphics[scale=.75]{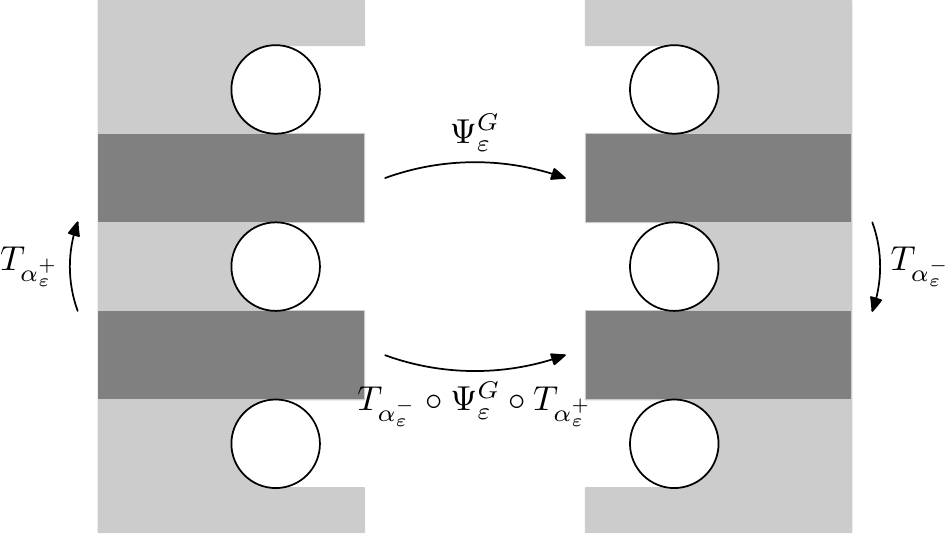}}\hfill
    \subcaptionbox{Covering of $D(0,r)\setminus\{\pm\sqrt\ep\}$
      by $S_\ep^+$ and $S_\ep^-$ and their intersection
      \label{fig:inter S}}
      [.45\textwidth]
      {\includegraphics[scale=1]{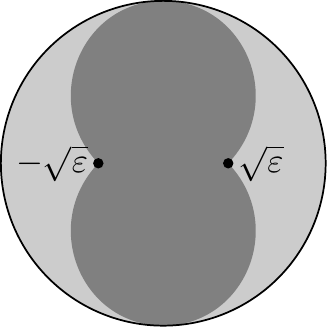}}
    \caption{Orbits covered twice.}
  \end{figure}

\subsection{Analytic Properties of the Transition Functions
and Normalization}

For $\ep\leq 0$, $\Psi_\ep^0$ is
determined by $\Psi_\ep^\infty$ in~\eqref{eq:PsiInf commute}.
Therefore, we only to describe the analytic properties of
$\Psi_\ep^\infty$ and $\Psi_\ep^G$. The properties are the
same as in the holomorphic case, with the exception of
Equation~\eqref{eq:termes const} below, which is specific to 
the antiholomorphic case.

\penalty-500
\begin{prop}\label{prop:prop analytiques} \hspace{0pt}
 \begin{enumerate}\item The transition functions of $f_\ep$ commute with
  $T_1$. They have a series expansion of the form
  \begin{equation}\label{eq:serie PsiInf}
    \Psi_\ep^\infty(W) = W + c_0^\infty(\ep) 
      + \sum_{n=1}^\infty c_n^\infty(\ep) e^{2i\pi n W},
  \end{equation}
  for $\ep \leq 0$, and 
  \begin{equation}\label{eq:serie PsiG}
    \Psi_\ep^G(W) = W + c_0^G(\ep) 
      + \sum_{n\in Z^\ast} c_n^G(\ep) e^{2i\pi nW},
  \end{equation}
  for $\ep > 0$.
  The constant terms satisfy
  \begin{equation}\label{eq:termes const}
    \Im c_0^\infty(\eps) = -i\pi b(\eps)
    \qquad\text{ and }\qquad
    \Im c_0^G(\eps) = -i\pi b(\eps).
    \end{equation}
    \item Moreover, there exists a choice of Fatou coordinates $\Phi_\ep^\pm$ for which
      the constant terms are purely imaginary:\begin{align}
    c_0^\infty(\ep) &= -i\pi b(\ep),
    \quad\text{for $\ep \leq 0$},\label{c_normalized_1}\\[2\jot]
    c_0^G(\ep) &= -i\pi b(\ep),
    \quad\text{for $\ep > 0$}.\label{c_normalized_2}
  \end{align}  
  
%\item The Lavaurs translation is given by\begin{equation}T^L=T_{-\frac{i\pi}{\sqrt{\eps}}}.\label{value_Lavaurs}\end{equation}
  \end{enumerate}
      \end{prop}
\begin{proof}
  Equations~\eqref{eq:serie PsiInf} and~\eqref{eq:serie PsiG} 
  follow from the fact that
  $\Psi^\infty_\ep$ and $\Psi_\ep^G$ commute with $T_1$.
  Furthermore, for $\ep \leq 0$,
  $c_n^\infty(\ep) = 0$ for $n< 0$ because $\Psi^\infty_\ep(W)-W$
  is bounded when $\Im W \to \infty$.

  For the proof of Equation~\eqref{eq:termes const} for
  $c_0^\infty$, we need the well-known equation 
  \begin{equation}
  c_0^\infty - c_0^0 =-2i\pi b.\label{c0-infty}
  \end{equation}
  This equation will be a direct consequence of Lemme~\ref{lem_orbite_neg}, 
  see Remark~\ref{rem:cinf 2ipi} below. By comparing the constant terms
  of~\eqref{eq:PsiInf commute}, we find $c_0^\infty = \ol{c_0^0}$. 
  Now Equation~\eqref{eq:termes const} for $c_0^\infty$ follows 
  by combining ~\eqref{c0-infty} with 
  $c_0^\infty = \ol{c_0^0}$.

For $c_0^G$, it follows by comparing the constant terms of
Equation~\eqref{eq:PsiG commute}.

It suffices to postcompose one Fatou coordinate with a real translation to obtain the normalizations \eqref{c_normalized_1} or \eqref{c_normalized_2}.
 \end{proof}

\begin{deff}\label{fatou_normalized}
  A family of transition functions 
  ${\{\Psi_\ep\}}_\ep$ is said to be
  \emph{normalized} if it is continuous in
  $(\ep,Z)$ and real analytic for $\ep\not=0$
  and if \eqref{c_normalized_1} or \eqref{c_normalized_2} is verified.
  The corresponding pairs of families of Fatou coordinates ${\{\Phi_\ep^\pm\}}_\ep$ are said to be \emph{normalized}.
\end{deff}

\subsection{Geometry of the Space of Orbits for $\ep > 0$}

\begin{theo}
  For $\ep>0$ the space of orbits $\joli O_\ep$ of $f_\ep$ 
  on $D(0,r)\setminus\{\pm\sqrt\ep\}$
  is described by two Klein bottles $K_{\ep}^\pm$,
  identified along Moebius strips $M_{\ep}^\pm$,
  by a map $[\Psi_\ep]$ induced by the transition function
  $\Psi_\ep$: see Figure~\ref{fig:klein}.

  We add two points $\{P_{\ep}^\pm\}$ to $\joli O_\ep$ corresponding
  to the fixed points $\pm\sqrt\ep$ with the following topology:
  the unique neighbourhood of $P_{\ep}^\pm$ is $K_{\ep}^\pm$.
\end{theo}

\begin{figure}[htbp]
  \centering
	\includegraphics{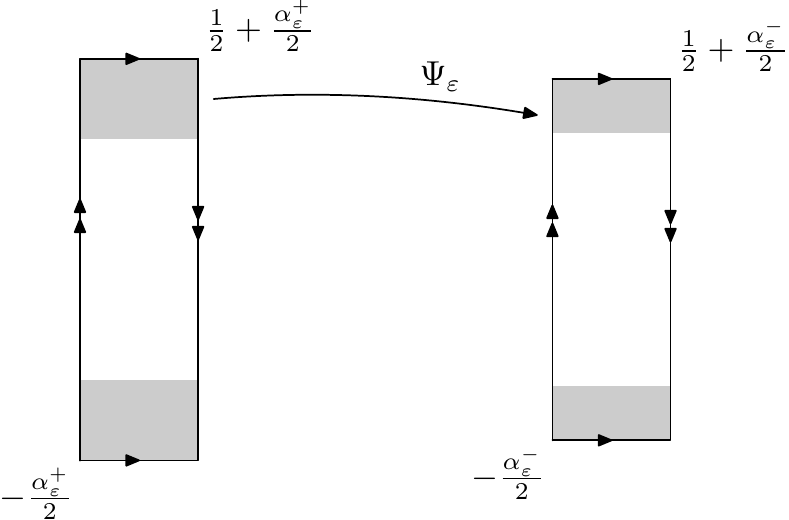}
  \caption{Klein bottles and transition function.}
	\label{fig:klein}
\end{figure}

\begin{proof}
  We set  $S_\ep^\pm=Z_\ep^{-1}(U_\ep^\pm)$, 
  which yields a covering of $D(0,r)\setminus\{\pm\sqrt\ep\}$: see 
  Figure~\ref{fig:inter S}.  
  Note that some orbits of $f_\ep$ in $D(0,r)\setminus\{\pm\sqrt\ep\}$,
 are covered by both charts.

 Let $V_\ep^\pm = \Phi_\ep^\pm(U_\ep^\pm)$. On $V_\ep^\pm$, the region of intersection corresponds to
  horizontal strips. The transition function $\Psi_\ep$ maps orbits from
  a strip in $V_\ep^+$ to the corresponding strip in $V_\ep^-$.
  Since the orbits repeat themselves according to the period
  $\alpha_\ep^\pm$ on $V_\ep^\pm$, the strips also repeat with
  associated transition functions. See Figure~\ref{fig:strips}.
  
  To describe the space of orbits of $f_\ep$, we begin by taking
  a vertical strip of width $1$ in both $V_\ep^+$ and $V_\ep^-$. Together, they
  intersect all the orbits of $f_\ep$ at least once, except the
  fixed points. We identify $W\in V_\ep^\pm$ with $T_1(W)$ and
  $T_{\alpha_\ep^\pm}(W)$, leaving us with two torii
  $[0,1]\times \big[-{\alpha_\ep^\pm\over 2}, {\alpha_\ep^\pm\over 2}\big]$.
  Since the transition function commutes with both $T_1$ and
  $T_{\alpha_\ep^\pm}$, it is well-defined on the torii. We identify
  together $w$ and $\Psi_\ep(w)$, as they represent the same orbit of
  $f_\ep$. This is the space of orbits of $f_\ep\circ f_\ep$. 
  
  Lastly, since $\STt$ is comptatible with $T_1$,  $T_{\alpha_\ep^\pm}$
  and $\Psi_\ep$,
  it induces a mapping on each torus. We identity $w$ and $\STt(w)$
  to obtain the space of orbits of $f_\ep$. Each torus becomes a 
  Klein bottle, as in Figure~\ref{fig:klein}.
\end{proof}

\subsection{Geometry of the Space of Orbits for $\ep < 0$}\label{sect:orbit_negative}

Note that the space of orbits of $f_\eps$ is a quotient of that of $g_\ep$. Hence we will describe both. This will need introducing a few notions.

 Let us consider two strips $B_{\ell_\pm}$ on each side of the fundamental hole, and let $\Ss^\pm$ be two spheres corresponding to the completion of the strips $B_{\ell_\pm}$ with sides identified. 
  Then each orbit of
  $f_\ep$ corresponds to at least one point of
  each sphere, with $\infty$ corresponding to $\sqrt\ep$, 
  and $0$ to $-\sqrt\ep$. Let $\psi_\ep^{0,\infty}$ be the
  mapping induced by $\Psi_\ep^{0,\infty}$ between neighbourhoods of $0$ and $\infty$ on the two spheres.
   
    \begin{figure}[htbp]
    \centering
    \includegraphics[scale=.75]{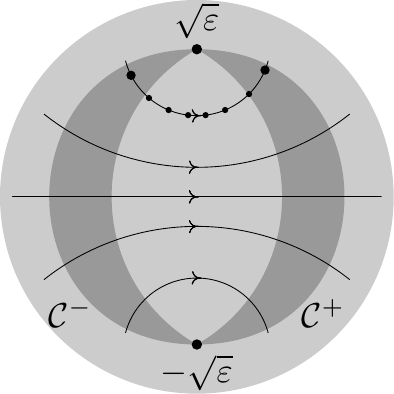}
    \caption{Point in $\mathcal{C}^-$ mapped to the point of its
    orbit in $\mathcal{C}^+$ by the Lavaurs transition.}
    \label{fig:translation lavaurs}
  \end{figure}

    \begin{lem}\label{lem_orbite_neg} Let $\ep<0$. 
 \begin{enumerate} 
    \item There exists a global map called the \emph{Lavaurs transition} 
       that maps an orbit of $f_\ep$ in $\Ss^-$
      to the same orbit in $\Ss^+$: see Figure~\ref{fig:translation lavaurs}. It is a linear map $\ell^L$ on the
      spheres, and a translation $T^L$ in the Fatou coordinates.
      When the Fatou coordinates are normalized, then \begin{equation}T^L=T_{-\frac{i\pi}{\sqrt{\eps}}}.\label{value_Lavaurs}\end{equation}
    \item The linear map $L_{-1}\circ \tau\colon \Ss^\pm\to \Ss^\pm$;
      $w\mapsto -{1\over \ol w}$
      is the action of $f_\ep$ on $\Ss^\pm$ 
      (see Figure~\ref{fig:orbite ep < 0}); it 
      maps an orbit of $f_\ep$ to the same orbit and the quotient
      $\Pp_\pm:= \Ss^\pm/ L_{-1}\circ \tau$ is a real projective plane;
    \item The maps $\kappa^{0,\infty}= \ell^L \circ \psi_\ep^{0,\infty}$ 
      are \emph{first return
      maps} of $g_\ep$ around $-\sqrt\ep$ and $+\sqrt\ep$ respectively and
      they are compatible with the orbits of $f_\ep$, i.e.~%
      $(L_{-1} \circ \tau) \circ (\ell^L \circ \psi_\ep^0)
        = (\ell^L \circ \psi_\ep^{\infty}) \circ (L_{-1} \circ \tau)$.
    \item In the time
      coordinate, the first return map of $\sqrt\ep$ (resp.~$-\sqrt\ep$)
      is $T_{-\alpha_\ep^+}$ (resp.~$T_{\alpha_\ep^-}$) as seen on 
      Figure~\ref{fig:retour}.
\item The first return maps can be written in the Fatou coordinate
  by
  \begin{align}\label{eq:premier retour inf}
    T^L \circ \Psi_\ep^\infty &= \Phi_\ep^+ \circ T_{-\alpha_\ep^+}
      \circ (\Phi_\ep^+)\inv,\\[2\jot]\label{eq:premier retour 0}
    T^L \circ \Psi_\ep^0 &= \Phi_\ep^+ \circ T_{\alpha_\ep^-}
      \circ (\Phi_\ep^+)\inv.
  \end{align}
  \end{enumerate} 
\end{lem}

\begin{figure}[htbp]
  \begin{center}  
  \includegraphics[width= 4cm]{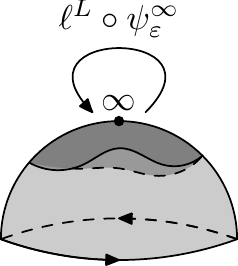}
  \caption{The space of orbits of $f_\ep$ for $\ep < 0$.}
  \label{fig:orbite ep < 0}
  \end{center}
\end{figure}

\begin{rem}\label{rem:cinf 2ipi}
  Comparing constant terms from Equations~\eqref{eq:premier retour inf}
  and~\eqref{eq:premier retour 0}, we find the well known relation
  $c_0^\infty - c_0^0 = -2i\pi b$ for $\ep < 0$. It holds for any transition functions,
  not necessarily normalized.
\end{rem}

The following theorem is a direct consequence of Lemma~\ref{lem_orbite_neg}.

 \begin{theo}\label{orbit-space_negative} The space of orbits $\joli O_\ep$ of $f_\ep$ is the quotient of a real projective plane by a diffeomorphism in the neighborhood of one point.
This diffeomorphism corresponds to a first return map in the neighborhood of the periodic points. See Figure~\ref{fig:orbite ep < 0}. 
\end{theo}

\begin{proof}[Proof of Lemma~\ref{lem_orbite_neg}]
  The space of orbits of $g_\ep$ can be described with a
  croissant, as in Figure~\ref{fig:croissant}, and the return maps around
  each fixed point (see \cite{germeDeploie}). If we quotient by $f_\ep$, we obtain the
  space of orbits of $f_\ep$. 

  On a Fatou coordinate $W = \Phi_\ep^+(Z)$ of $f_\ep$, 
  we take a strip of width~1 to the left of the fundamental hole. 
  With the universal covering $E: W\mapsto \exp(-2i\pi W)$,
  we see that this strip is isomorphic to $\Cp^\ast$. We
  identify $\sqrt\ep$ to $\infty$ and $-\sqrt\ep$ to $0$,
  so that we now have a sphere. Each point represents an
  orbit of $g_\ep$. On the sphere, $f_\ep$ becomes $L_{-1}\circ \tau$.

  \begin{figure}[htbp]
    \centering
    \subcaptionbox{First return map on the croissant\label{fig:croissant}}
      [.35\textwidth]
      {\includegraphics[scale=.75]{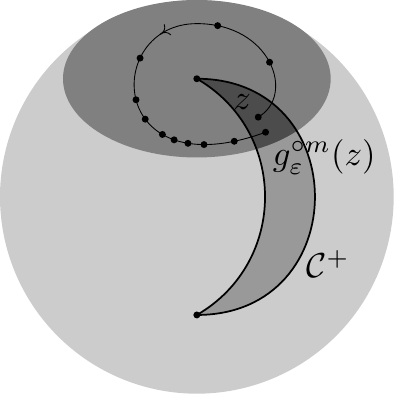}}\hfil\hfil
    \subcaptionbox{First return map in the time coordinate,
     with $B=Z_\ep^+(\mathcal{C}^+)$\label{fig:retour}}
      {\includegraphics[scale=.75]{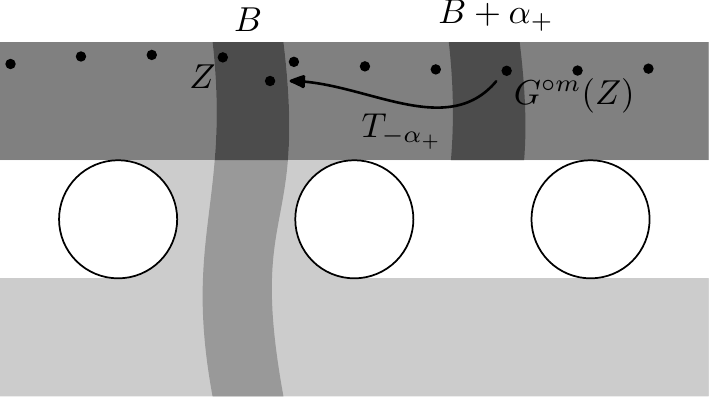}}
    \caption{Croissant $\mathcal{C}^+$ and first return map.}
  \end{figure}

  In the $z$-coordinate, the sphere
  can be seen as a croissant $\mathcal{C}^+$ going from $-\sqrt\ep$ to
  $\sqrt\ep$, see Figure~\ref{fig:croissant}.
  We see that every orbit of $f_\ep$
  intersects the croissant. In the neighborhood of each fixed point $\pm\sqrt\ep$, 
  we can define a first return map $p^\pm$ in the croissant
  for the orbits of $g_\ep$. In the time coordinate,
  the first return maps are $T_{-\alpha_\ep^+}$ and $T_{\alpha_\ep^-}$
  (see Figure~\ref{fig:retour}),
  so we see that they are compatible with the orbits of $f_\ep$
  since $T_{-\alpha_\ep^+} \circ F_\ep = F_\ep \circ T_{-\ol{\alpha_\ep^+}}$
  and $-\ol{\alpha_\ep^+} = \alpha_\ep^-$.

  We consider a strip of width~1 \relax 
  in the Fatou coordinate of $\Phi_\ep^-$ to the right of the fundamental hole, yielding a second croissant $\mathcal{C}^-$ in the $z$-coordinate.
  We define a diffeomorphism that maps each point
  of the croissant $\mathcal{C}^-$ to the first point of its forward orbit in
  the croissant $\mathcal{C}^+$, see Figure~\ref{fig:translation lavaurs}. 
  This map defines
  a global diffeomorphism from a sphere to another sphere,
  so it is a linear map $\ell^L$. In the Fatou coordinates, it is
  a translation; we call it the \emph{Lavaurs translation}
  $T^L$. Combining~\eqref{def_psi} and~\eqref{eq:premier retour inf}
  and using the definition of $\alpha_\ep^+$~\eqref{eq:periode},
  we obtain $T^L = \Phi^-_\ep \circ T_{-{i\pi\over\sqrt \ep}} \circ (\Phi_\ep^+)\inv$.
  Then because $T^L$ is a translation and $\Psi^\infty_\ep$ is
  normalized~\eqref{c_normalized_1}, we obtain $T^L =T_{-\frac{i\pi}{\sqrt{\eps}}}$.
  
  Let $\psi_\ep^{0,\infty}$ be the maps induced on the spheres
  by $\Psi_\ep^{0,\infty}$.
  The space of orbits of $g_\ep$ is obtained by identifying
  $w\in \Ss^-$ with $\ell^L(w)\in \Ss^+$, and $w\in \Ss^+$ with $\psi_\ep^{0,\infty}(w)\in\Ss^-$.
  Since $\ell^L$ and $\psi_\ep^{0,\infty}$ are compatible with
  the orbits of $f_\ep$, $f_\ep$ induces the global 
  diffeomorphism $L_{-1}\circ\tau$ on the space of orbits of $g_\ep$.
  We identify $w$ with $L_{-1}\circ \tau(w)$ to obtain
  the space of orbits of $f_\ep$: this yields a real projective plane
  with identification of $w$ and $\psi_\ep^\infty(w)$.
 \end{proof}

\section{Weak Classification}\label{sec:weak class}
We have all the tools to define the \emph{weak modulus of 
classification}. 

The goal is to prove the strong equivalence
(Definition~\ref{def:eq forte}) of two
families with the same weak modulus. With the 
tools we have so far, we can only prove a weaker
form of equivalence (see Definition~\ref{def:weak_equivalence} below). Indeed, Fatou coordinates,
instrumental in the construction of the equivalence,
are not analytic at $\ep=0$. 

In Section~\ref{sec:weak class theo}, we will give a simple
proof of weak equivalence and we will prove
strong equivalence in Section~\ref{sec:strong class}.

\subsection{Weak modulus of Classification}
The space of orbits of $f_\ep$ can be described
by the codimension, the formal invariant and
one transition function. These are the three parts
of the weak modulus of classification. 

\begin{deff}\label{def:weak_equivalence}
  Let $f_\eta$ be a generic unfolding of a
  parabolic antiholomorphic germ of codimension~1.
  Its \emph{weak modulus of classification} is the
  triple $(\ep, b, [\Psi_\ep])$, where $\ep$
  is the canonical parameter of $f_\eta$, $b$
  is the formal invariant (a real analytic function of $\ep$) and $[\Psi_\ep]$
  is an equivalence class of normalized families of transition functions
  under the relation $\sim$
  $$
    {\{\Psi_\ep\}}_\ep\sim {\{\Psi_\ep'\}}_\ep \iff
    \begin{aligned}
      &\exists C(\ep)\in\R\text{ with $C$ real analytic for}\cr
      &\text{$\ep\not=0$ and continuous such that} \cr
      &\Psi_\ep' = T_{C(\ep)} \circ \Psi_\ep \circ T_{-C(\ep)}.
    \end{aligned}
  $$ 
\end{deff}

\subsection{Weak Classification Theorem}
\label{sec:weak class theo}

%\begin{deff}[Weak Equivalence] Let $f_{1,\ep}$
%  and $f_{2,\eta}$ be two unfoldings of 
%  antiholomorphic parabolic germs of codimension~1.
%  They are \emph{weakly equivalent} if there
%  exist $r,r'>0$ and a diffemorphism
%  \begin{align*}
%    H\colon (-r',r')\times D(0,r) &\to \R\times \Cp\\
%      (\ep,z) &\mapsto (\beta(\ep), h_\ep(z))
%  \end{align*}
%  such that

\begin{deff}[Weak Equivalence]
  Let $f_{1,\eta}$ and $f_{2,\ep}$ be two generic unfoldings
  of antiholomorphic parabolic germs of codimension $1$. We say they are
  \emph{weakly equivalent} if there exists an open interval
  $I\ni 0$, a disc $D(0,r)$, $r>0$, and a mix analytic
  diffeomorphism $H\colon I\times D(0,r) \to \R\times \Cp$
  such that
  \begin{enumerate}
    \item $H(0,0) = (0,0)$;
    \item $H(\eta,z) = (\beta(\eta), h_\eta(z))$ with 
      $\beta$ real analytic and $h_\eta$ continuous in
      $(\eta,z)$ and real analytic for $\eta\not=0$;
    \item $f_{2,\beta(\eta)} = h_\eta \circ f_{1,\eta} \circ h_{\eta}\inv$.
  \end{enumerate}
\end{deff}

\begin{theo}[Weak Classification]\label{theo:weak eq}
  Two generic unfoldings of antiholomorphic parabolic germs 
  of codimension~1 are weakly equivalent if and 
  only if they have the same weak modulus of classification.
\end{theo}
\begin{proof}
  One direction is obvious. Conversely, let
  $f_{1,\ep}$ and $f_{2,\ep}$ be the two families that, without loss of generality, we can suppose 
 in prepared form, and let $\ep$ be their canonical parameter. Moreover, let us
  suppose that the two families have the same weak modulus
  $(\ep,b,[\Psi_\ep])$.
 Then $\beta\equiv {\rm id}$. 
   
  We use Fatou coordinates $\Phi_{j,\ep}^\pm$ of $f_{j,\ep}$
  to construct the change of coordinate $h_\ep$
  that conjugates $f_{1,\ep}$ to $f_{2,\ep}$. 
  Since $f_j$ have the same weak modulus, we can choose 
  $\Phi_j^\pm$ so that 
$$
    \Phi_{1,\ep}^- \circ T_{-i\pi b}\circ (\Phi_{1,\ep}^+)\inv
    = \Psi_\ep
    = \Phi_{2,\ep}^- \circ T_{-i\pi b}\circ (\Phi_{2,\ep}^+)\inv.  $$

  We can divide $D(0,r)\setminus\{\pm\sqrt\ep\}$ in
  two regions $S_\ep^\pm$ as in Figures~\ref{fig:R pm ep < 0} 
  and~\ref{fig:R pm ep > 0}. For $z\in S_\ep^\pm$,
  we define $h_\ep^\pm$ by
  \begin{equation}\label{eq:def h ep}
    h_\ep^\pm(z) = 
      (Z^\pm_{\ep})\inv \circ (\Phi_{1,\ep}^\pm)\inv 
        \circ \Phi_{2,\ep}^\pm \circ Z^\pm_{\ep}(z).
  \end{equation}
  A direct computation shows that 
  $(h_\ep^\pm)\inv \circ f_{1,\ep} \circ h_\ep^\pm = f_{2,\ep}$.
  So it only remains to prove that 
  $h_\ep^+ = h_\ep^-$ on $S_\ep^+\cap S_\ep^-$ in order to define
  $h_\ep$ on a $D(0,r)$. 

  \begin{figure}[htbp]
  \centering
    \subcaptionbox{${S}_\ep^-$}
      {\includegraphics{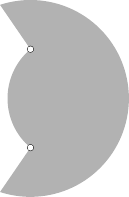}}\hfil
    \subcaptionbox{${S}_\ep^+$}
      {\includegraphics{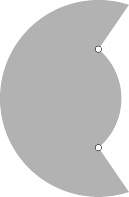}}\hfil
    \subcaptionbox{${S}_\ep^-\cup {S}_\ep^+$}
      {\includegraphics{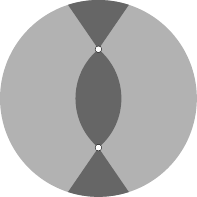}}
    \caption{Division of $D(0,r)$ for $\ep < 0$.}
    \label{fig:R pm ep < 0}
    \vskip2em
    \subcaptionbox{${S}_\ep^-$}
      {\includegraphics{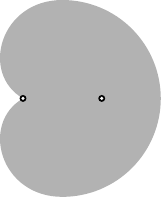}}\hfil
    \subcaptionbox{${S}_\ep^+$}
      {\includegraphics{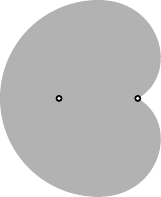}}\hfil
    \subcaptionbox{${S}_\ep^-\cup {S}_\ep^+$}
      {\includegraphics{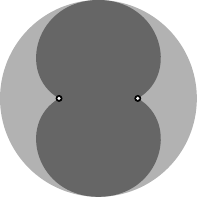}}
    \caption{Division of $D(0,r)$ for $\ep > 0$.}
    \label{fig:R pm ep > 0}
  \end{figure}

  For $\ep > 0$, we have, using \eqref{eq:coordonnee temps_neg} and \eqref{def_psi}:
  \begin{equation}\label{eq:calcul h pm}
  \begin{aligned}
		h_\ep^-\circ (h_\ep^+)\inv
		  &= \Big((Z_\ep^-)\inv\circ (\Phi_{2,\ep}^-)\inv\circ \Phi_{1,\ep}^-
			  \circ Z_\ep^-\Big) \\
				&\qquad\quad \circ \Big( (Z_\ep^+)\inv
			  \circ (\Phi_{1,\ep}^+)\inv\circ \Phi_{2,\ep}^+\circ Z_\ep^+\Big)\\[2\jot]
		  &= (Z_\ep^-)\inv \circ (\Phi_{2,\ep}^-)\inv\circ \Psi^G_\ep\circ \Phi_{2,\ep}^+
			  \circ Z_\ep^+\\[2\jot]
		  &= (Z_\ep^-)\inv \circ T_{-i\pi b}
			  \circ Z_\ep^+\\[2\jot]
			&= id.
  \end{aligned}
  \end{equation}

  For $\ep<0$, there are three cases. The intersection
  ${S}_\ep^+\cap {S}_\ep^-$ has three components: $I^+$ above 
  $\sqrt\ep$, $I^-$ below $-\sqrt\ep$ and $I^L$ between
  $\sqrt\ep$ and $-\sqrt\ep$ (see Figure~\ref{fig:R pm ep < 0}). For $I^\pm$, the computations
  are the same as in~\eqref{eq:calcul h pm}, but with
  $\Psi^{0,\infty}$. For $I^L$, the computation is
  similar, but we use the fact that $f_1$ and $f_2$
  have the same Lavaurs translation (this is true as soon
  as $f_1$ and $f_2$ have the same formal invariant, see \eqref{value_Lavaurs}).

  We define
  $$
    h_\ep(z) = \begin{cases}
      h_\ep^+(z), & \text{if $z\in R_\ep^+$;}\\
      h_\ep^-(z), & \text{if $z\in R_\ep^-$;}
    \end{cases}
  $$
  then $h_\ep$ is well-defined on $D(0,r)\setminus\{\pm\sqrt\ep\}$
  by the above and we can analytically continue 
  $h_\ep$ on $\pm\sqrt\ep$, since it bounded.
\end{proof}

\section{Strong Classification}
\label{sec:strong class}
The Weak Equivalence Theorem~\ref{theo:weak eq} has
a simple and direct proof. The main result of this
section is that two families with same weak modulus of classification are strongly equivalent (see Definition~\ref{def:eq forte}). The proof
is more involved and will necessitate to introduce the notion of \emph{strong modulus of classification}.

Let $f_\ep$ be an unfolding of an antiholomorphic parabolic
germ in prepared form. In the proof of Theorem~\ref{theo:weak eq},
we constructed a weak equivalence $h_\ep$ using normalized Fatou coordinates
of $f_\ep$. The idea is to complexify $\ep$ and to continue analytically the Fatou coordinates in a normalized way: they will be ramified in $\ep$. 
The analytic extension of $h_\ep$ will be shown to be holomorphic in $\ep$ around $\ep=0$ in the nontrivial case when the transition maps are generically not translations, while in the trivial case, we will force $h_\ep$ to be  holomorphic using \emph{strongly normalized Fatou coordinates} (defined below).

\subsection{Complex Parameter}
Recall that we work with a representative of the
germ $f_\ep$ defined on $(-r',r')\times D(0,r)$.
We complexify $\ep$ and we continue analytically $f_\ep$
so that it is antiholomorphic in $\ep$. Now 
$$g_\ep = f_\epbar \circ f_\ep$$ is a full holomorphic unfolding
of $g_0$. 

Since $b$ is real-valued for real values of $\ep$, it
commutes with $\sigma$. It follows that $\ol{v_\ep(z)} = v_\epbar(\zbar)$,
for the vector field $v_\ep(z) = {z^2 - \ep \over 1 + b(\ep)z}$,
and therefore the time-$t$ satisfies 
$\sigma \circ v_\ep^t = v_\epbar^{\ol t}\circ \sigma$.

The time coordinate is also defined for complex
values of $\ep$. However, the line of the holes, given
by the period $\alpha_\ep^\pm$, rotates when $\ep$
varies around $0$. See Figure~\ref{fig:ep rotation}. 
Using the same definitions as in Proposition~\ref{prop:time} we see that the two charts in time $Z_\ep^\pm$ cannot be defined uniformly in time. 
Therefore, we lift
$\ep$ on the universal covering of $D(0,r')^\ast$
and we work on a sector of the form
(see Figure~\ref{fig:secteur ep})
\begin{equation}
  \Opi := \{\hatep = \rho e^{i\theta}\mid 0 < \rho <r',\ 
    -\pi + \delta < \theta < 3\pi - \delta\}.
\end{equation} Then $Z_\hatep^\pm$ are well defined on $\Opi$. 

\begin{figure}[htbp]
  \centering
  \includegraphics{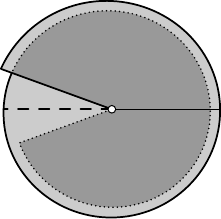}
  \caption{Sector $\Opi$.}
  \label{fig:secteur ep}
\end{figure}

Let us generalize the translation domains of $g_\eps$ to complex
values of $\ep$. Let $G_\hatep$ be the lift of $g_\ep$ on the time
coordinate. The bound $|G_\hatep - T_1| \leq C \max\{r,r'\}$ 
still holds for $\hatep \in\Opi$, for some constant $C$~\cite{germeDeploie}. So
we may take $r,r'$ small enough so that $G_\hatep$ and
$T_1$ are as close as needed.
We take a line $\ell$ that is transversal to 
the line of the holes such that $\ell$ and $G_\hatep(\ell)$
do not intersect each other and any of the holes (see Figure~\ref{fig:ep rotation}). 
Then we set $B_\ell$ to be region between $\ell$
and $G_\hatep(\ell)$, including the boundary, and we define
the translation domain
$$
  U_\hatep = \{G_\hatep^{\circ n}(Z) \mid Z\in B_\ell\}.
$$
Note that for $\arg\hatep =\pi$, this corresponds to the
Lavaurs translation domain defined in 
Definition~\ref{def:trans dom}.

We see that there are two non equivalent ways to choose $\ell$
for $\hatep$ such that $-\pi+\delta < \arg\hatep < \pi-\delta$
(see third row in Figure~\ref{fig:ep rotation} for $\ep >0$),
which explains why we choose $\Opi$ with an overlapping over these
values.

\begin{figure}[htbp]
  \centering
	\centerline{
    \subcaptionbox{Rotation of the holes when $\ep$ goes around~$0$
      clockwise\label{fig:ep tourne horaire}}
      [.5\textwidth]{
	    \centering
	  	\includegraphics{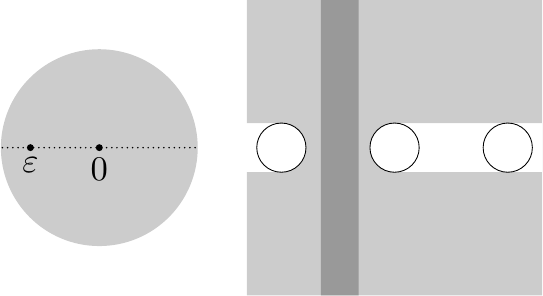}
	  	\vskip2em
	  	\includegraphics{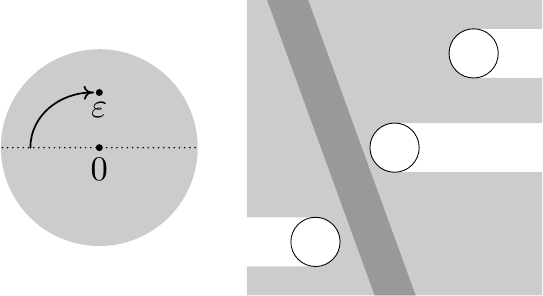}
      \vskip2em
	  	\includegraphics{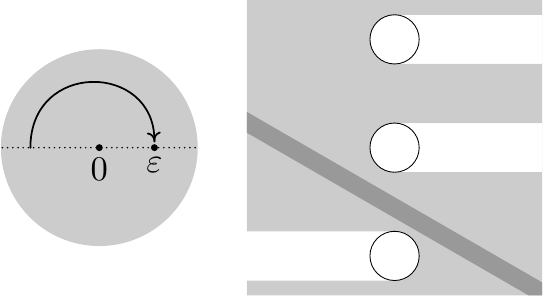}
      \vskip2em
	  	\includegraphics{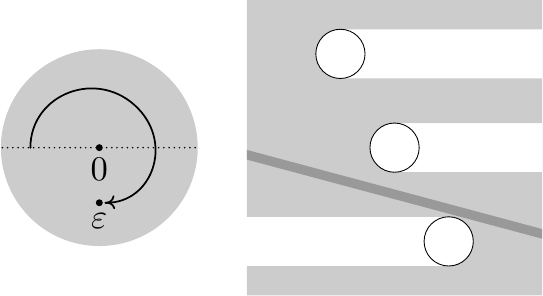}
      }
	  \hskip5em
    \subcaptionbox{Rotation of the holes when $\ep$ goes around~$0$
      counter-clockwise\label{fig:ep tourne antihoraire}}
      [.5\textwidth]{
	    \centering
	  	\includegraphics{figures/fatou-7}
	  	\vskip2em
	    \includegraphics{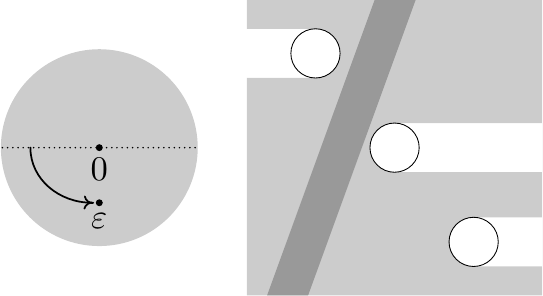}
	  	\vskip2em
	    \includegraphics{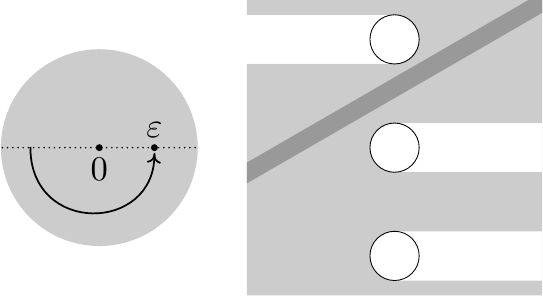}
	  	\vskip2em
	    \includegraphics{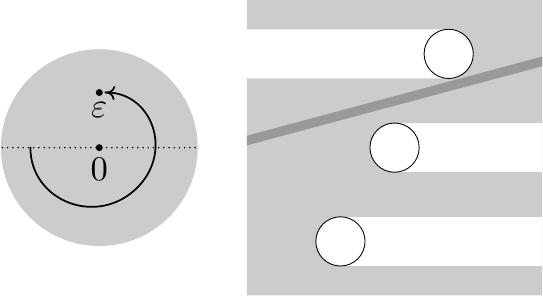}	
      }
  }
  \caption{Time coordinates and strips $B_\ell$ when  $\ep$ rotates from $\mathbb R^-$.}
  \label{fig:ep rotation}
\end{figure}

\begin{lem}[\cite{germeDeploie}] 
  Let $\hatep$ be the lift of $\ep$ on the universal
  covering of $D(0,r')^\ast$. If $r,r'$ are small enough, 
  then there exists a continuous
  family of translation domains ${\{U_{\hatep}\}}_{\hatep\in\Opi}$.
\end{lem}

On the overlapping sector $-\pi + \delta < \arg\hatep < \pi - \delta$,
both $\hatep$ and $e^{2i\pi}\hatep$ project on the same
point $\ep$. However, if ${\{U_{\hatep}\}}_{\hatep\in\Opi}$
is a continous family, then $U_{\hatep} \not= U_{e^{2i\pi}\hatep}$,
as seen on Figure~\ref{fig:ep rotation}.

\subsection{Fatou Coordinates and Transition Functions for
  a Complex Parameter}

We define a complex conjugate on $\Opi$. In a small sector centered on the negative real axis, i.e.~centered on $\arg\hatep = \pi$,
we define it as the reflection along this axis
and then we anti-holomorphically continue it on $\Opi$.
This gives us
\begin{equation}
  \arg \ol{\hatep\,} = 2\pi - \arg\hatep.
\end{equation}

\begin{prop}[Continuation of Fatou Coordinates]\label{prop:fatou ep complexe}
  Let $f_{\hatep}$ be the family induced on $\Opi$ by $f_\ep$
  and let $F_\hatep$ be the lift of $f_\hatep$ on the time
  coordinate.
  \begin{enumerate}
   \item The set 
      $$
        Q^\pm = \bigcup_{\hatep \in\Opi} \{\hatep\}\times U_\hatep^\pm
      $$
      is a complex manifold of dimension~2, where 
      ${\{U_\hatep^\pm\}}_{\hatep\in\Opi}$ is a continuous
      family of translation domains. 
    \item There exists a family 
      $\Phi^\pm = {\{\Phi_\hatep^\pm\}}_\hatep$
      of Fatou coordinates of $f_\hatep$ such that \begin{itemize}
  \item ${\{\Phi_\hatep^\pm\}}_\hatep$ satisfies 
    \begin{equation}\label{eq:F STt}
    \Phi_\hatepbar^\pm \circ F_\hatep\circ (\Phi_\hatep^\pm)\inv
      = \STt;    \end{equation}
  \item $\Phi^\pm$ is holomorphic on $Q^\pm$ and continuous at $\hatep=0$, i.e.~
    $$
      \lim_{\hatep\to 0\atop \hatep\in\Opi} 
      \Phi^\pm(\hatep,\cdot) = \Phi_0^\pm(\cdot),
    $$
    where the convergence is uniform on compact sets
    and $\Phi_0^\pm$ is a Fatou coordinate of $f_0$
    on $U_0^\pm$; 
    \item The family is uniquely determined
    by
    \begin{equation}\label{eq:determine Phi}
      \ol{\Phi_\hatepbar^\pm(X_\hatepbar{^\pm})} + \Phi_\hatep^\pm(X_\hatep{^\pm}) 
        = C{^\pm}(\hatep),
    \end{equation}
    where $X_\hatep{^\pm}$ is a base point, $C{^\pm}$ commutes with 
    $\sigma$ and both $X_\hatep{^\pm}$ and $C{^\pm}$ are holomorphic in $\hatep\neq0$ with continuous limit at $\hatep=0$.
\end{itemize}
   % \item Let ${\{(\Phi_\hatep^\pm)'\}}_{\hatep\in\Opi}$ be another family of Fatou coordinates as in Point 2, then there exists a
 %   family of constants $\{{\color{red}D}_\hatep{\color{red}^\pm}\in\Cp\}_{\hatep\in \Opi}$ such that ${\color{red}D}_{\hatepbar}{\color{red}^\pm}=\overline{{\color{red}D}_{\hatep}{\color{red}^\pm}}$
 %   and \begin{equation} (\Phi_\hatep^\pm)'  = T_{{\color{red}D}_\hatep{\color{red}^\pm}} \circ \Phi_\hatep^\pm.\end{equation} 
\end{enumerate}
\end{prop}
\begin{proof}
The idea for the proof of Point 1 is that for $(\hatep_0,Z_0)\in Q^\pm$ 
with $Z_0\in B_\ell$, there exists a bi-disk 
$D(\hatep_0,r_1)\times D(Z_0,r_2)\subset Q^\pm$ for
some $r_1,r_2>0$ small enough. Then for other 
points $(\hatep_0,Z)\in Q^\pm$, there is a neighborhood of
the form 
$D(\hatep_0, r_1')\times G^{\circ n}_{\hatep_0}\big(D(Z_0,r_2')\big)\subset Q^\pm$,
for $r_1', r_2' >0$ small enough and $n\in\Z$ and $Z_0\in B_\ell$ such
that $G_{\hatep_0}^{\circ n}(Z_0) = Z$. (More details in \cite{germeDeploie}).

For Point 2, we drop the upper indices $\pm$. It is known (\cite{germeDeploie}) that there exist Fatou coordinates  ${\{\wt\Phi_\hatep\}}_{\hatep\in \Opi}$ for $g_\ep$ holomorphic 
in $\hatep$ with continuous limit at $\ep=0$. Let \begin{equation}\widetilde{P}_\hatep=\wt\Phi_\hatepbar\circ F_\hatep\circ (\wt\Phi_\hatep)^{-1}.\label{def:P}\end{equation}
Then $\widetilde{P}_\hatepbar\,\circ \widetilde{P}_\hatep=T_1$, from which it follows that $\widetilde{P}_\hatep\,\circ T_1=T_1\circ \widetilde{P}_\hatep$. As in Theorem~\ref{theo:coord fatou} it follows that $\widetilde{P}_\hatep=\Sigma\circ T_{C(\hatep)}$, with  $C(\hatepbar) + \ol{C(\hatep)} = 1$ and $C$ depending analytically on $\hatep$ with continuous limit at $\ep=0$.   We set $\Phi_\hatep = T_{-\ol{C(\hatepbar)}} \circ \wt\Phi_\hatep$.
%  these are Fatou coordinates of $f_\ep$ holomorphic in $\hatep$. 
 
 To see that~\eqref{eq:determine Phi} determines ${\{\Phi_\hatep\}}_\hatep$,
suppose ${\{\Phi^\dag_\hatep\}}_\hatep$ is another family such
that
$$
  \ol{\Phi_\hatepbar(X_\hatepbar)} + \Phi_\hatep(X_\hatep)
  = 
  \ol{\Phi_\hatepbar^\dag(X_\hatepbar)} + \Phi_\hatep^\dag(X_\hatep).
$$
By uniqueness of Fatou coordinates of $g_\ep$ up to translation, there exists $D$ such that
$D(\hatepbar) = \ol{D(\hatep)}$ and 
$\Phi_\hatep = T_{D(\hatep)} \circ \Phi_\hatep^\dag$.
Substituting this in the previous equation yields
$D(\hatep)=0$.
%Point 3 follows from the uniqueness up to translation of the
%Fatou coordinates and having to preserve~\eqref{eq:F STt}.
\end{proof}

\begin{prop}[Continuation of the Transition Functions]
\label{prop:continuation Psi}
Let ${\{\Phi_\hatep^\pm\}}_\hatep$ be two families of
Fatou coordinates of $f_\ep$ on $U_\hatep^\pm$.
We define the \emph{transition functions for 
$\hatep\in\Opi$} by
\begin{align}
  \Psi_\hatep^\infty(W) &:= \Phi_\hatep^- \circ 
    T_{-i\pi b(\hatep)} \circ (\Phi_\hatep^+)(W),\\[2\jot]
  \Psi_\hatep^0(W) &:= \Phi_\hatep^- \circ 
    T_{i\pi b(\hatep)} \circ (\Phi_\hatep^+)(W),
\end{align}
where the composition is defined respectively above
or below the fundamental hole.
Then we have
\begin{enumerate}
  \item $\Psi_\hatep^\infty$, for $\arg\hatep = \pi$, coincides
    with $\Psi_\ep^\infty$, for $\ep < 0$;
  \item $\Psi_\hatep^\infty$ commutes with $T_1$; 
    \item $\Psi_\hatep^\infty$ and $\Psi_\hatep^0$ satisfy
    \begin{equation}\label{eq:Psi commute hatep}
      \STt \circ \Psi_\hatep^\infty = \Psi_\hatepbar^0 \circ \STt;
    \end{equation}
  \item $\Psi_\hatep^\infty$ has the series expansion
    \begin{equation}\label{eq:serie PsiInf hatep}
      \Psi_\hatep^\infty(W) = W + c_0^\infty(\hatep) + 
        \sum_{n=1}^\infty c_n^\infty(\hatep) e^{2i\pi n W};
    \end{equation}
  \item If ${\{\Phi_\hatep^\pm\}}_\hatep$ are families holomorphic
    in $\hatep\not=0$ with a limit as $\hatep\to 0$,
    then $\Psi_\hatep^\infty$ is holomorphic in $\hatep\not=0$
    and has a limit when $\hatep\to 0$;

  \item If we choose the base point $X_\hatep^+ = Z_\hatep^+(r) = 0$, 
    then the family ${\{\Phi_\hatep^+\}}_\hatep$ determined by
    $\Phi_\hatep^+(X_\hatep^+)= 0$ is a family of Fatou coordinates 
    of $f_\ep$ depending analytically of $\hatep\neq0$ with continuous limit at $\hatep=0$. We can then further choose the second family of Fatou coordinates so that the transition functions be normalized, namely so that the constant term of $c_0^\infty(\hatep)$ in \eqref{eq:serie PsiInf hatep} is given by $c_0^\infty(\hatep)= -i\pi b(\hatep)$. Then $\Phi_\hatep^-$ also depends analytically of $\hatep$ with continuous limit at $\hatep=0$.
\end{enumerate}
\end{prop}

\begin{proof}

Point 1 follows from the fact that the Lavaurs translation
domains $U_\ep^\pm$ for $\ep<0$ coincide with the translation
domains $U_\hatep^\pm$ for $\arg\hatep = \pi$, and hence so do the
Fatou coordinates and therefore, the transition functions. 

Equation~\eqref{eq:Psi commute hatep} follows
directly from~\eqref{eq:F STt}.

The proof of~\eqref{eq:serie PsiInf hatep}
is indentical to the case of the Lavaurs transition function for $\ep < 0$.
See~Proposition~\ref{prop:prop analytiques}.

  Point 5 is obvious.
  
  Lastly, for point 6, we see that the families 
  ${\{\Phi_\hatep^\pm\}}_\hatep$ both satisfy some condition of the form~\eqref{eq:determine Phi} with $C$ depending analytically on $\hatep$ with continuous limit at $\hatep =0$.
  \end{proof}

\begin{deff}\label{def:normalized fatou}
  We say that a pair of families of Fatou coordinates ${\{\Phi_\hatep^\pm\}}_\hatep$
  is strongly normalized if they are chosen as described in point 6 of 
  Proposition~\ref{prop:continuation Psi}.  When this is the case,
  the transition functions are said to be \emph{\emph strongly normalized}.
\end{deff}

\paragraph{Continuation of the Lavaurs Translation.}
The Lavaurs translation was defined in Lemma~\ref{lem_orbite_neg}. It can be extended for $\hatep\in\Opi$ in the following way.
We consider two vertical strips in the Fatou coordinates,
one on the left and one on the right of the fundamental
hole, see Figure~\ref{fig:ep rotation} for the strip on the left.
In the $z$-coordinate, they correspond to two croissants,
as in Figure~\ref{fig:croissant ep}.
This allows us to define the Lavaurs translation for
$\hatep\in \Opi$. Indeed, since each orbit intersects both
croissants at least once, we can define it the same way we
did for $\ep < 0$. 

\begin{figure}[htbp]
  \centering
	\includegraphics[scale=.8]{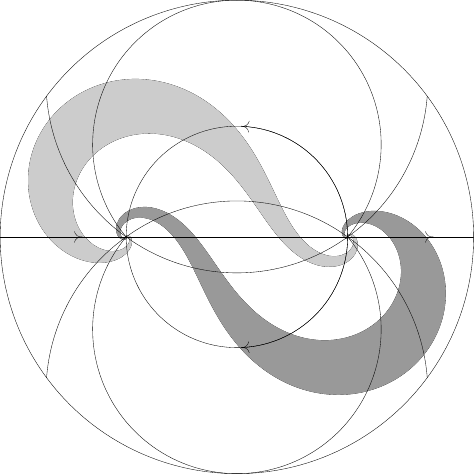}
	\qquad\qquad
  \includegraphics[scale=.9]{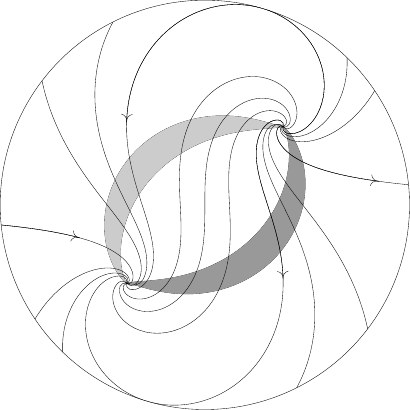}
  \caption{Croissants for $\ep = {1\over 4}$ and
  $\ep = {e^{i\pi/2}\over 4}$ in the case of the
  normal form, with $b=0$.}
  \label{fig:croissant ep}
\end{figure}

\subsection{The strong modulus of analytic classification} 

\begin{deff}\label{def:strong_modulus}
  Let $f_\eta$ be a generic unfolding of a
  parabolic antiholomorphic germ of codimension~1.
  Its \emph{strong modulus of classification} is the
  triple $(\ep, b, [\{\Psi^\infty_\hatep\}_{\hatep\in\Opi}])$, where $\ep$
  is the canonical parameter of $f_\eta$, $b$
  is the formal invariant (a real analytic function of $\ep$) and $[\{\Psi^\infty_\hatep\}_{\hatep\in\Opi}]$
  is an equivalence class of normalized families of transition functions
  under the relation $\sim$
  $$
    {\{\Psi_\hatep^\infty\}}_{\hatep}\sim {\{\widetilde{\Psi}_\hatep^\infty\}}_{\hatep} \iff
    \begin{aligned}
      &\exists C(\hatep) \text{ where $C$ is analytic for $\hatep\not=0$}\cr
      &  \text{with continuous limit at $\ep=0$, }\cr
      &{\ol{C(\hatep)} = C(\hatepbar),}\cr
      &\widetilde{\Psi}_\hatep^\infty = T_{C(\hatep)} \circ \Psi_\hatep^\infty \circ T_{-C(\hatep)}.
    \end{aligned}
  $$ 
\end{deff}

\subsection{The Strong Classification Theorem}\label{sec:proof strong class}
Recall that the strong equivalence corresponds
to Definition~\ref{def:eq forte}. 
\begin{theo}[Strong Classification]\label{theo:strong class}
 The following are equivalent:
 \begin{description}
 \item{(1)} Two generic unfoldings of antiholomorphic parabolic germs 
  of codimension~1 are strongly equivalent.
  \item{(2)} They have the same weak modulus of classification.
  \item{(3)} They have the same strong modulus of classification.\end{description}\end{theo}
\begin{proof} We have seen that (1) implies (2) implies (3). Let us now show that (3) implies (1). 

Let ${\{f_{j,\ep}\}}_\ep$, $j=1,2$, be two families unfolding 
antiholomorphic parabolic germs of codimension~1 with the same strong modulus.
They induce families ${\{f_{j,\hatep}\}}_\hatep$ with
$\hatep\in \Opi$. 

We will give different arguments when the strong modulus is nontrivial and when it is trivial. 

\medskip

\noindent{\bf The strong modulus is nontrivial.} This means that $\Psi_{j,\hatep}^\infty$ is generically not a translation. 
We choose for each family ${\{f_{j,\ep}\}}_\ep$ a pair of normalized Fatou coordinates (see Definition~\ref{fatou_normalized}) so that the corresponding transition functions are equal: $\Psi_{1,\hatep}^{0,\infty}\equiv \Psi_{2,\hatep}^{0,\infty}$. Since the strong moduli are equal, the Lavaurs translations are equal. 

We define a change of coordinate $h_\hatep$ the same way 
we did for $\ep < 0$ in~\eqref{eq:def h ep}, namely by 
 \begin{equation}\label{eq:def h hatep}
    h_\hatep^\pm(z) = 
      (Z_\hatep^\pm)\inv \circ (\Phi_{2,\hatep}^\pm)\inv 
        \circ \Phi_{1,\hatep}^\pm \circ Z_\hatep^\pm(z),
  \end{equation}
   on two domains $S_\hatep^\pm$ covering $D(0,r)$. The domains are taken as in Figure~\ref{domain_S}. They are projections of domains $R_\hatep^\pm$ as in Figure~\ref{domain_R}. 
As in Theorem~\ref{theo:weak eq} we can show that $h_\hatep^+=h_\hatep^-$ on the intersection of their domains, yielding that $h_{\hatep}$ is well defined. Moreover, $f_{2,\hatep} = h_\hatepbar \circ f_{1,\hatep} 
\circ h_\hatep\inv$.  In particular this implies 
\begin{equation}
  g_{2,\hatep} = h_\hatep \circ g_{1,\hatep} 
  \circ h_\hatep\inv.\label{conj_g1_g2}
\end{equation}

\begin{figure}[htbp]
  \centering
  \includegraphics[scale=.8]{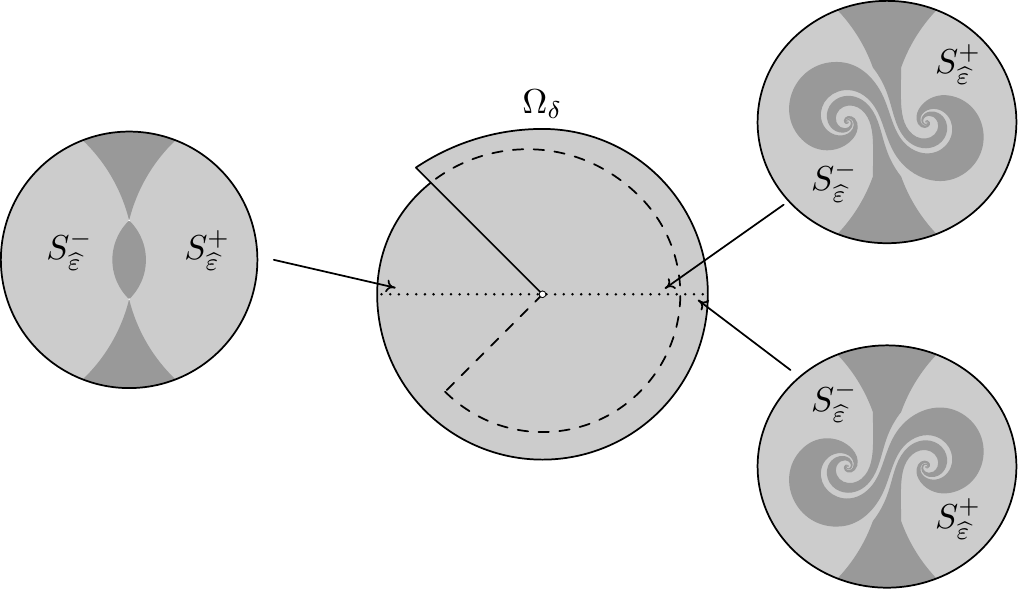}
  \caption{Sectors $S_\hatep^\pm$ for $\arg\hatep=\pi$ (on the left),
    $\arg\hatep = 0$ (on the upper right) and $\arg\hatep=2\pi$ (on the lower right).}
  \label{domain_S}
\end{figure}

\begin{figure}
  \begin{center} 
  \includegraphics{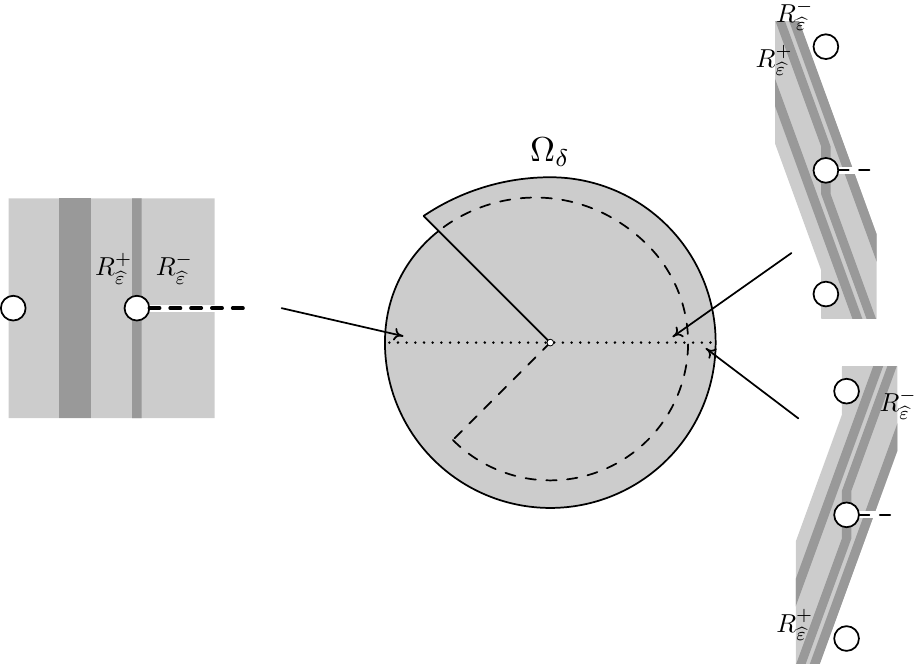}\par
  \caption{Domains of definition of return maps. The domains $R_\hatep^\pm$ projecting on the sectors $S_\hatep^\pm$ of Figure~\ref{domain_S}.}
  \label{domain_R}
  \end{center}
\end{figure}

We already know that $h_\hatep$ is holomorphic in $\hatep$ and we need to show that it is uniform in $\ep$. 
Recall that there is an overlapping region in $\Opi$.
On the sector $-\pi + \delta < \arg\hatep < \pi - \delta$,
both $\hatep$ and $e^{2i\pi}\hatep$ in $\Opi$ project on the same
point $\ep\in D(0,r')^\ast$. 
Let us prove that $h_\hatep = h_{e^{2i\pi}\hatep}$
for $-\pi + \delta < \arg\hatep < \pi - \delta$.

Let 
\begin{equation}k_\hatep=(h_{e^{2i\pi}\hatep})^{-1} \circ h_\hatep.\label{def:k}\end{equation}
Hence for $\hatep$ in the self-intersection 
of $ \Opi$ it follows from \eqref{conj_g1_g2} that 
\begin{equation}g_{1,\hatep}\circ k_\hatep =k_\hatep \circ g_{1,\hatep}.\label{eq:k}\end{equation}
Since $k_\hatep$ commutes with $g_{1\ep}$, by Lemma~\ref{lem:symmetries} below there exists a rational number $m$ such that $k_\hatep =g_{1,\hatep}^{\circ m}$. But $\lim_{\hatep\to 0}k_\hatep={\rm id}$. It follows that $m=0$ and $k_\hatep={\rm id}$.

\medskip

\noindent{\bf The strong modulus is trivial.} Then all $\Psi_{j,\hatep}^{0,\infty}$ are translations. In particular, for all normalized Fatou coordinates of 
${\{\Phi_{j,\hatep}^\pm\}}_\hatep$
of $f_{j,\hatep}$ 
we have that $\Psi_{1,\hatep}^{0,\infty}\equiv \Psi_{2,\hatep}^{0,\infty}$. 
We choose strongly normalized Fatou coordinates (see Definition~\ref{def:normalized fatou}).

We use the change of coordinate $h_\hatep$ defined in \eqref{eq:def h hatep}, which yields \eqref{conj_g1_g2}.
For $-\pi + \delta < \arg\hatep < \pi - \delta$, we define again for $k_\ep$ as in \eqref{def:k}, which satisfies \eqref{eq:k}. 
Let $K_\ep^\pm$ the expression of $k_\hatep$ in the Fatou coordinates.
In these coordinates $K_\ep^\pm$ is a translation, that must commute with the transition maps. Moreover, $K_\ep^+(0)=0$ because of the strong normalization of the Fatou coordinates. Hence $K_\ep^+=id$, which yields that $k_\hatep=id$ on the corresponding domain in $z$-space, and then on the whole disk by analytic continuation.
\end{proof}

\begin{lem}\label{lem:symmetries}
Let $g_\ep$ be a generic family unfolding a holomorphic parabolic germ, and $\{(\Psi_\hatep^0,\Psi_\hatep^\infty)\}_{\hatep\in \Opi}$
 be a family of transition  functions  of $g_\eps$. Let $\Omega$ be a connected subset of parameter space with at least one accumulation point, and let $\{h_\ep\}_{\ep\in \Omega}$ be an analytic family of diffeomorphims over $D(0,r)$ commuting with $g_\ep$.  If at least one of $\Psi_\hatep^0$ or $\Psi_\hatep^\infty$ is not a translation for at least
 one $\ep\in\Omega$, then there exists a rational number $m$ such that $h_\ep=g_\ep^{m}$ for all $\ep\in \Omega$.
\end{lem}
\begin{rem} If $\Psi_\hatep^\infty$ or $\Psi_\hatep^0$ is not
  a translation for at least one $\ep\in\Omega$, then by the
  identity principal, it is not a translation over all $\Omega$,
  except perhaps on a discrete set of points.
\end{rem}
\begin{proof}
Let $$H_\ep^\pm= \Phi_\ep^\pm \circ Z_\ep^\pm \circ h_\ep\circ(Z_\ep^\pm)\inv\circ (\Phi_\ep^\pm)\inv.$$ 
Then $H_\ep^\pm$ commutes with $T_1$ over $U^\pm$, which means that $H_\ep^\pm$ is a translation $T_{m_\hatep^\pm}$, where $m_\hatep^\pm$ depends analytically on $\eps\in \Omega$. Moreover, we need have $$\Psi_\hatep^{0,\infty}\circ T_{m_\hatep^+} = T_{m_\hatep^-}\circ \Psi_\hatep^{0,\infty}.$$ From the form \eqref{eq:serie PsiInf hatep}
 of $\Psi_\hatep^{\infty}$ and the similar form for $\Psi_\hatep^0$, it follows that $m_\hatep^+=m_\hatep^-:=m_\hatep$. Moreover, if one of $\Psi_\hatep^0$ or $\Psi_\hatep^\infty$ is not a translation, then $m_\hatep$ is a rational number for almost all $\ep\in \Omega$, yielding that $m_\hatep$ is constant. \end{proof}

\section{Applications and Realisation}\label{sec:app}

\subsection{Antiholomorphic Square Root}
%\subsubsection{Existence of antiholomorphic square root of a parabolic holomorphic germ}

\begin{theo}[Antiholomorphic Square-Root]\label{thm:square_root}
  Let $g_\ep$ be an unfolding of a holomorphic parabolic germ of
  codimension~1 depending holomorphically on the complex parameter $\ep$. 
  There exists a family $\{f_\ep\}_{\ep}$ depending anti-holomorphically 
  of $\ep$ such that 
  \begin{equation}f_\epbar \circ f_\ep = g_\ep\label{def:g_f_anti}\end{equation} if and only if there exists families ${\{\Psi_\hatep^{0,\infty}\}}_\hatep$ of transition functions of $g_\ep$ such that
    for all $\hatep\in \Opi$
  \begin{equation}\label{eq:Psi racine antihol}
    \STt \circ \Psi_\hatep^\infty = \Psi_\hatepbar^0 \circ \STt.
  \end{equation}
  In particular, if $\ep$ is real, then $f_\ep\circ f_\ep= g_\ep$. 
\end{theo}

\begin{proof}
  If $f_\ep$ exists, then we must have~\eqref{eq:Psi racine antihol}.

  Now suppose~\eqref{eq:Psi racine antihol} is true. We can suppose
  that $\Psi_\hatep^\infty$ is non-linear, since the linear case is trivial.
  Let  ${\{\Phi_\hatep\}}_\hatep$ be normalized Fatou coordinates of
  $g_\ep$ associated to the transition functions. Let $S_\hatep^\pm$ be
  as in the proof of Theorem~\ref{theo:strong class}
 (Section~\ref{sec:proof strong class}), see 
  Figure~\ref{domain_S}. We define $f^\pm_\hatep$ by   
  $$f^\pm_\hatep=(Z_\hatepbar^\pm)\inv\circ (\Phi_\hatepbar^\pm)\inv \circ \STt \circ \Phi_\hatep^\pm\circ Z_\hatep^\pm.$$
 Then
$f_\hatep^\pm$ is anti-holomorphic on ${S}_\hatep^\pm$.
  On the intersection
  of ${S}_\hatep^+ \cap {S}_\hatep^-$, we have that $f_\hatep^+ \circ (f_\hatep^-)\inv = id$,
  because of~\eqref{eq:Psi racine antihol}. 
    Indeed, the argument is the same as in the proof of
 Theorem~\ref{theo:weak eq}. 

  Moreover, $f_\hatep$
  is defined so that $g_\ep = f_\hatepbar \circ f_\hatep$.
  Let $\gamma_\ep = f_{\hatep e^{2i\pi}}\inv 
  \circ f_\hatep$. 
  When $\arg\hatep=0$, then $g_\ep= f_\hatepbar \circ f_\hatep = f_\hatep \circ f_\hatepbar$ and $f_\hatepbar=f_{\hatep e^{2i\pi}}$. Hence $\gamma_\ep$ commutes with $g_\ep$. Since $\Psi_\hatep^\infty$ is non linear 
  it follows by Lemma~\ref{lem:symmetries} that $\gamma_\ep = g_\eps^{\circ m}$ for some rational $m$,
  and since $\gamma_\ep \to {\rm id}$ when $\ep\to 0$, we have $\gamma_\ep = {\rm id}$ for $\arg\hatep=0$. 
  By the identity principle it follows that $\gamma_\ep={\rm id}$ for $-\pi+\delta < \arg\hatep < \pi-\delta$.
  We conclude that $f_\hatep$ is uniform in $\ep$.
  Since $f_\ep$ is anti-holomorphic in $\ep$, it follows that 
  it is real analytic in $\ep$ for $\ep$ real.
\end{proof}

A necessary condition for a holomorphic parabolic germ $g_0$ to have an antiholomorphic square root  is that $b\in\mathbb R$ and that there exist Fatou coordinates for which the transition functions $\Psi_0^0$ and $\Psi_0^\infty$ are linked by the equation
\begin{equation}\Psi_0^0\circ \Sigma\circ T_{\frac12} = \Sigma\circ T_{\frac12}\circ \Psi_0^\infty. \label{relation_Psi}\end{equation}  The latter is a condition of infinite codimension. 
Now that we have studied the unfolding we can explain this condition. Indeed, when we perturb $g_0$ to some $g_\eps$ with $b(\eps)\in \mathbb R$, we can see that the dynamics near the fixed points $\pm\sqrt{\ep}$ is given by the first return maps. The interesting values of $\ep$ are the ones for which the multiplier has modulus $1$ of the form $\exp(2\pi i \beta)$. Indeed, in that case the singular points are generically nonlinearizable as soon as $\beta$ is either rational, or does not satisfy Bruno condition. When $\beta$ is rational and the fixed point is not linearizable, then the first return map 
has a formal invariant and a functional modulus. And when $\beta$ is irrational there may be, as described by Yoccoz and Perez-Marco, accumulation of periodic points or hedgehog dynamics in the neighborhood of the fixed point. 

Suppose now that the unfolding $g_\ep$ is of the form $f_\ep^{\circ 2}$ for some antiholomorphic parabolic unfolding $f_\ep$. Then, when the product of the multipliers at the fixed points of $g_\ep$ has modulus $1$ and is different from $1$, the two fixed points form a periodic orbit of period $2$ of $f_\ep$. This means that the first return maps in the neighborhood of the two fixed points must be conjugate for all values of $\ep$. For instance, just considering the rational values of $\beta$, the first return maps must have the same codimension and the same formal invariants, as well as the same analytic parts of the modulus. This is obviously a very strong condition and it should be no surprise that this is not possible when \eqref{relation_Psi} is not satisfied. 

\subsection{Realisation}\label{sec:realisation}

A bonus of Theorem~\ref{thm:square_root} is to provide a necessary and sufficient condition for a strong modulus \begin{equation}(\ep, b, [\{\Psi_\hatep^\infty\}_{\hatep\in \Opi}]).\label{modulus_antiholomorphic}\end{equation}
 to be realised as the strong modulus of a generic family $f_\ep$ unfolding an antiholomorphic parabolic germ.  As a comparison, the modulus of a germ of generic analytic family unfolding a holomorphic parabolic point has the form
\begin{equation}(\ep, b,  [\{(\Psi_\hatep^0,\Psi_\hatep^\infty)\}_{\hatep\in \Opi}]).\label{modulus_parabolic}\end{equation}
The realisation problem was solved for such germs (see \cite{realisation}). Hence, taking a triple of the form \eqref{modulus_antiholomorphic}, we extend it to a triple of the form \eqref{modulus_parabolic} using \eqref{eq:Psi racine antihol}. Then the triple \eqref{modulus_antiholomorphic} is realisable as an antiholomorphic germ of family if and only if the extended triple \eqref{modulus_parabolic} is realisable as a holomorphic germ of family, which has a square root. 
The latter is the case if and only if the extended triple satisfies some \emph{compatibility condition}. This condition states that the dynamics given by $(\Psi_\hatep^0,\Psi_\hatep^\infty)$ and by $(\Psi_{\hatep e^{2i\pi}}^0,\Psi_{\hatep e^{2i\pi}}^\infty)$ are conjugate. It takes a special form in the case of antiholomorphic germs of families that we now state. 
\bigskip

\noindent{\bf The compatibility condition in the antiholomorphic case.} We have seen that for $\ep>0$ the dynamics can be described by the \emph{Glutsyuk modulus} $\Psi_\ep^G$, where we bring the family to the normal form in the neighborhood of each fixed point through the Glutsyuk Fatou coordinates and we compare the normalizations. This description can of course be extended in the whole overlapping part of the projection of $\Opi$ in $\ep$-space.
In the Glutsyuk Fatou coordinates, the first return maps (or their inverses) are simply given by $T_{-\alpha_\hatep^\pm}$, where $\alpha_\hatep^\pm$ is defined as 
\begin{equation}\label{eq:periode_hat}
  \alpha_\hatep^\pm := \int\limits_{{\gamma_\pm}} {1 + b(\ep)\zeta \over \zeta^2 - \ep}\d\zeta
    = \pm {i\pi\over \sqrt{\hatep}} + i\pi b(\ep),
\end{equation} where $\gamma_\pm$ is a small loop surrounding $\pm \sqrt{\hatep}$ in the positive direction. 
Note that 
\begin{equation} \begin{cases}\alpha_{\hatep e^{2i\pi}}^\pm = \alpha_\hatep^\mp,\\[2\jot]
\Sigma(\alpha_{\hatepbar}^\pm) =- \alpha_\hatep^\mp.\end{cases}\label{properties:alpha}\end{equation}

We compare the Glutsyuk Fatou coordinates obtained through straightening the first return maps over the principal holes (the hatched regions in Figure~\ref{fig:domains_comparison}). This in turn forces the choice of  the chosen first return maps (or their inverses) around both singular points. For all $\hatep$, $\Psi_\hatep^\infty$ (resp. $\Psi_\hatep^0$) is associated to $\sqrt{\hatep}$ (resp. $-\sqrt{\hatep}$).

\begin{figure}
  \begin{center} 
  \includegraphics{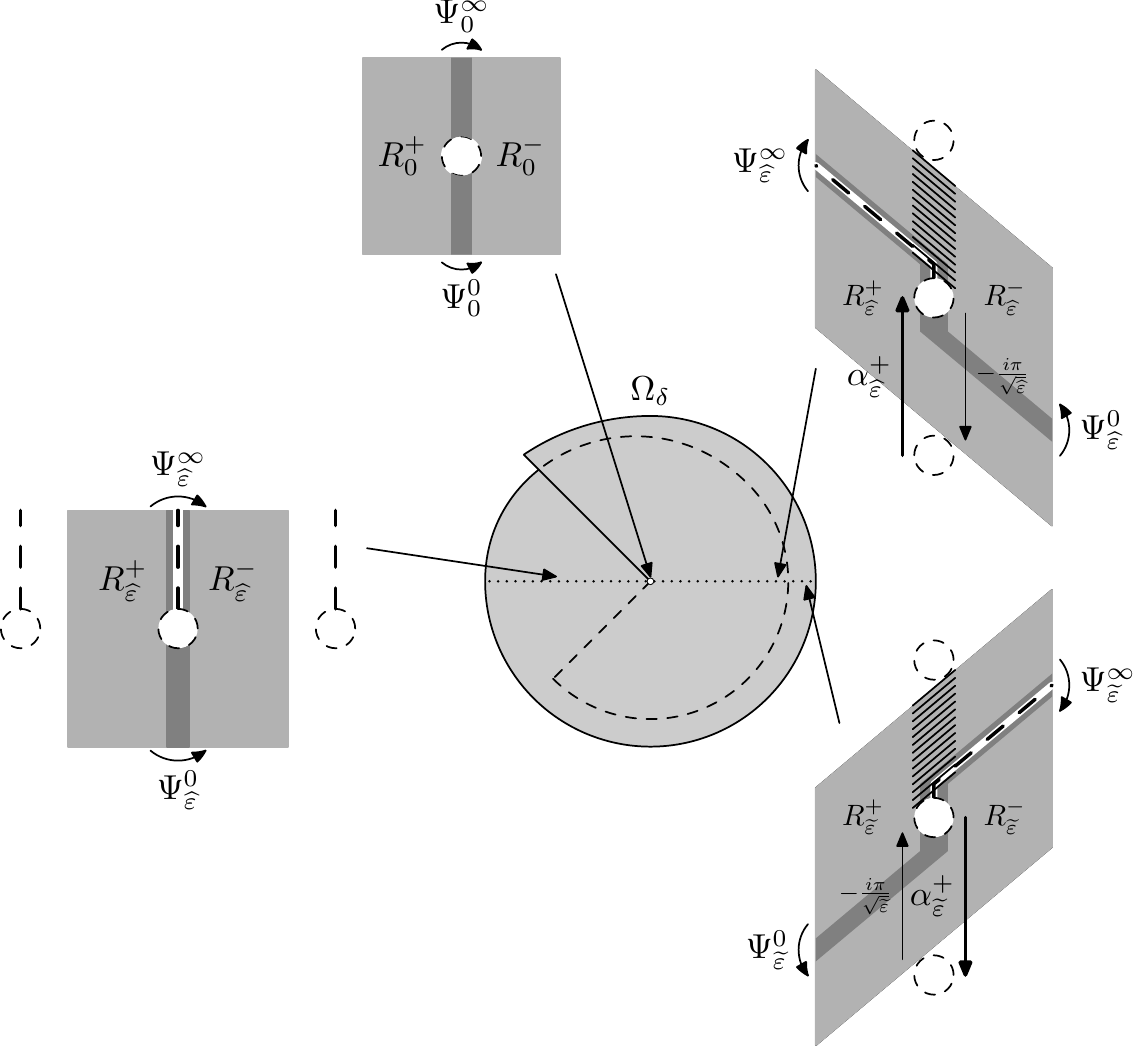}
  \caption{Domains of definition of return maps. }
  \label{fig:domains_comparison}
  \end{center}
\end{figure}

When $\arg\hatep\in(-\pi+\delta, \pi-\delta)$ we use the return maps $\Psi^0_\hatep\circ T_{-\frac{i\pi}{\sqrt{\hatep}}}$ and $\Psi^\infty_\hatep\circ T_{-\frac{i\pi}{\sqrt{\hatep}}}$, while for $\arg\hatep\in(\pi+\delta, 3\pi-\delta)$
we use $T_{-\frac{i\pi}{\sqrt{\hatep}}}\circ \Psi^0_\hatep$ and $T_{-\frac{i\pi}{\sqrt{\hatep}}}\circ \Psi^\infty_\hatep$.

\begin{theo}[Compatibility Condition \cite{realisation}] We consider a germ of generic analytic family $g_\ep$ unfolding a holomorphic parabolic germ and its modulus $(\ep, b, [\{(\Psi_\hatep^0,\allowbreak \Psi_\hatep^\infty)\}_{\hatep\in \Opi}])$.
Let us denote $\wep= \hatep e^{2i\pi}$. For $\arg\hatep\in(-\pi+\delta, \pi-\delta)$, let $H_\hatep^0$, $H_\hatep^\infty$, $H_{\wep}^0$, $H_{\wep}^\infty$ be defined as follows
\begin{align}\begin{split}
&H_\hatep^0\circ \Psi^0_\hatep\circ T_{-\frac{i\pi}{\sqrt{\hatep}}}= T_{\alpha_\hatep^-}\circ H_\hatep^0,\\
&H_\hatep^\infty\circ \Psi^\infty_\hatep\circ T_{-\frac{i\pi}{\sqrt{\hatep}}}= T_{-\alpha_\hatep^+}\circ H_\hatep^\infty,\\
&H_{\wep}^0\circ T_{-\frac{i\pi}{\sqrt{{\wep}}}}\circ\Psi^0_{\wep}= T_{\alpha_{\wep}^-}\circ H_{\wep}^0,\\
&H_{\wep}^\infty\circ T_{-\frac{i\pi}{\sqrt{{\wep}}}}\circ \Psi^\infty_{\wep}= T_{-\alpha_{\wep}^+}\circ H_{\wep}^\infty,\end{split}\label{def:H}\end{align}
and uniquely determined by
\begin{align}\begin{split}
&\lim_{\Im W\to+\infty} (H_\hatep^\infty-{\rm id})=0,\\
&\lim_{\Im W\to+\infty} (H_\wep^\infty-{\rm id})=0,\\
&\lim_{\Im W\to-\infty} (H_\hatep^0-{\rm id})=0,\\
&\lim_{\Im W\to-\infty} (H_\wep^0-{\rm id})=0.\label{limit_H}
\end{split}\end{align}
Then there exist $D_\ep, D_\ep'$ for $\arg\ep\in(-\pi+\delta, \pi-\delta)$ such that the following \emph{compatibility condition} is satisfied
\begin{equation} H_{\wep}^\infty\circ (H_{\wep}^0)\inv= T_{D_\ep}\circ H_\hatep^0\circ (H_\hatep^\infty)\inv\circ T_{D_\ep'}.\label{compatibility}\end{equation}
\end{theo}

\begin{theo}[Realisation Theorem \cite{realisation}]\label{theo:real_hol} Let be given a family $(\ep, b, [\{(\Psi_\hatep^\infty, \allowbreak \Psi_\hatep^0)\}_{\hatep\in \Opi}])$, where $b$ is an analytic function of $\ep$ defined on the projection of $\Opi$ and such that $\Psi_\hatep^{0,\infty}$ depend analytically on $\hatep\in \Opi$ with continuous limit when $\hatep\to 0$. Let $H_\hatep^0$, $H_\hatep^\infty$, $H_{\wep}^0$, $H_{\wep}^\infty$ be defined in \eqref{def:H} and satifying the compatibility condition \eqref{compatibility} together with \eqref{limit_H}. Then there exists a germ of generic analytic family $g_\ep$ unfolding a holomorphic parabolic germ realizing this modulus. 
\end{theo}

In the case of a germ of parabolic family of the form $g_\ep=f_\epbar\circ f_\ep$ we have the additional condition \eqref{eq:Psi commute hatep}, under which the compatibility condition takes a certain form. 
Indeed, we have 
\begin{align}
H_\hatep^0&=T_{\frac{i\pi}{\sqrt\hatep}}\circ T_{\frac12}\circ \Sigma\circ H_{\ol\hatep}^\infty\circ \Sigma\circ T_{-\frac12}\circ T_{-\frac{i\pi}{\sqrt{\hatep}}},\label{H_hatep^0}\\
H_\wep^0&= T_{\frac{i\pi}{\sqrt{\hatep}}}\circ  T_{\frac12}\circ \Sigma\circ H_{\ol\wep}^\infty\circ \Sigma\circ T_{-\frac12}\circ T_{-\frac{i\pi}{\sqrt{\hatep}}}.\label{H_wep^0}
\end{align}
If we let $N_{\hatep}=T_{\frac{i\pi}{\sqrt{\hatep}}} \circ T_{\frac12}\circ \Sigma$, the compatibility condition then takes the form: there exist $D_\ep, D_\ep'$ for $\arg\ep\in (-\pi + \delta, \pi - \delta)$ such that
\begin{equation}
  H_\wep^\infty \circ N_{\hatep} \circ (H_{\ol\wep}^\infty)\inv \circ N_{\hatep}\inv
    = T_{D_\ep} \circ N_{\hatep} \circ H_{\ol\hatep}^\infty \circ N_{\hatep}\inv
      \circ (H_\hatep^\infty)\inv \circ T_{D_\ep'}.
\label{compatibility_antihol}\end{equation}

\begin{theo}[Realisation Theorem in the anti-holomorphic case] Let be given a family $(\ep, b, [\{\Psi_\hatep^\infty\}_{\hatep\in \Opi}])$, where $b$ is an analytic function of $\ep$ defined on the projection of $\Opi$ and such that $\Psi_\hatep^\infty$ depends analytically on $\hatep\in \Opi$ with continuous limit when $\hatep\to 0$. Let $\Psi_\hatep^0$ be defined through \eqref{eq:Psi commute hatep}. Let $H_\hatep^\infty$ be defined by
\begin{equation}
  \begin{cases}
    H_\hatep^\infty\circ \Psi^\infty_\hatep\circ T_{-\frac{i\pi}{\sqrt{\hatep}}}= T_{-\alpha_\hatep^+}\circ H_\hatep^\infty,
      & \text{if $\arg\hatep\in (-\pi+\delta,\pi -\delta)$;}\\[2\jot]
    H_{\hatep}^\infty\circ T_{-\frac{i\pi}{\sqrt{{\hatep}}}}\circ \Psi^\infty_{\hatep}= T_{-\alpha_{\hatep}^+}\circ H_{\hatep}^\infty,
      & \text{if $\arg\hatep\in (\pi + \delta, 3\pi - \delta)$;}
  \end{cases}
\label{nouv_def_H}\end{equation}
and satisfying 
$$
  \lim_{\Im W\to +\infty} (H_\hatep^\infty - {\rm id}) = 0. 
$$
%\eqref{limit_H}. 
Suppose $H_{\hatep}^{\infty}$  satisfies the antiholomorphic compatibility condition, namely there exist $D_\ep, D_\ep'$ for $\arg\ep\in(-\pi+\delta, \pi-\delta)$ such that 
\begin{equation}
H_{\hatep e^{2i\pi}}^\infty \circ N_\hatep \circ (H_{\ol{\hatep e^{2i\pi}}}^\infty)\inv \circ N_\hatep\inv
  = T_{D_\ep} \circ N_\hatep \circ H_{\ol\hatep}^\infty \circ N_\hatep\inv
  \circ (H_\hatep^\infty)\inv \circ T_{D_\hatep'}.
\label{compatibility_antihol_bis}\end{equation}
Then there exists a germ of generic antiholomorphic family $f_\ep$ unfolding an antiholomorphic parabolic germ and realizing this strong modulus. 
\end{theo}
\begin{proof} We want to realize $f_\ep$. We first realize $g_\ep$ with modulus $(\ep, b, [\{(\Psi_\hatep^\infty,\allowbreak\Psi_\hatep^0)\}_{\hatep\in \Opi}])$, and then extract $f_\ep$ as an antiholomorphic square root satisfying $g_\ep:=f_\epbar\circ f_\ep$ using Theorem~\ref{thm:square_root}. 
Let us denote $\wep = \hatep e^{2i\pi}$. The functions $H_\hatep^0$ and $H_\wep^0$, defined in~\eqref{H_hatep^0} and~\eqref{H_wep^0} respectively, and $H_\hatep^\infty$ and $H_\wep^\infty$, defined by~\eqref{nouv_def_H}, satisfy \eqref{def:H} and \eqref{compatibility}, so by Theorem~\ref{theo:real_hol} we can realize $g_\ep$ and then take its antiholomorphic square root by Theorem~\ref{thm:square_root}. \end{proof}

\begin{rem}
For $\arg\hatep = 0$ and $\arg\wep = 2\pi$, the compatibility condition~\eqref{compatibility_antihol} becomes 
\begin{equation}
H_\wep^\infty\circ N_\ep\circ (H_\hatep^\infty)\inv \circ N_\ep\inv=T_{D_\ep}\circ N_\ep \circ H_\wep^\infty\circ N_\ep\inv\circ(H_\hatep^\infty)\inv\circ T_{D_\ep'}.\label{compatibility_antihol_0}\end{equation}
Indeed, in this cas $\ol{\hatep} = \wep$. This is a necessary condition
for the realisation.
\end{rem}

%In the case of a germ of parabolic family of the form $g_\ep=f_\epbar\circ f_\ep$ we have the additional condition \eqref{eq:Psi commute hatep}, under which the compatibility condition takes a certain form. 
%Indeed, when $\arg \hatep=0$, we have 
%\begin{align}H_\hatep^0&= T_{\frac12}\circ \Sigma\circ T_{-\frac{i\pi}{\sqrt{\hatep}}}\circ H_\wep^\infty\circ T_{\frac{i\pi}{\sqrt{\hatep}}}\circ \Sigma\circ T_{-\frac12},\label{H_hatep^0}\\
%H_\wep^0&=T_{\frac12}\circ \Sigma\circ T_{-\frac{i\pi}{\sqrt{\hatep}}}\circ H_\hatep^\infty\circ T_{\frac{i\pi}{\sqrt{\hatep}}}\circ \Sigma\circ T_{-\frac12}.\label{H_wep^0}\end{align}
%If we let $N_\ep=T_{\frac12}\circ \Sigma\circ T_{-\frac{i\pi}{\sqrt{\hatep}}}$, the compatibility condition then takes the form: there exist $D_\ep, D_\ep'$ for $\arg\ep=0$ such that
%\begin{equation}
%H_\wep^\infty\circ N_\hatep\circ (H_\hatep^\infty)\inv \circ N_\hatep\inv=T_{D_\ep}\circ N_\hatep \circ H_\wep^\infty\circ N_\hatep\inv\circ(H_\hatep^\infty)\inv\circ T_{D_\ep'}.\label{compatibility_antihol}\end{equation}

\subsection{Germs of Families with an Invariant Real Analytic Curve}

It was shown in \cite{GR} that an antiholomorphic parabolic germ preserves a real analytic curve if and only if $\Psi_0^\infty$ commutes with $T_{\frac12}$, which is a condition of infinite codimension. This means that it is exceptional that an antiholomorphic parabolic germ preserves a real analytic curve. Nonetheless, the real axis is invariant in  the formal normal form. 
Where is the obstruction?

Let us look at a generic unfolding. For $\eps>0$, the germ is analyticallly conjugate to the normal form in the neighborhood of each fixed point $\pm\sqrt{\ep}$ or, equivalently conjugate to $\sigma$ composed with the time-$\frac12$ of $\dot z = \lambda_\pm (z\mp\sqrt{\ep})$, where $\lambda_\pm=\frac{2\sqrt{\ep}}{1+b(\ep)\sqrt{\ep}}.$
The flow lines of the vector field are of the form $\cos\theta\, y-\sin\theta\, (x\pm \sqrt{\ep})=0$ for some $\theta\in [0,\pi)$. 
Among these flow lines exactly one is fixed by $\sigma$. In a generic unfolding there is no reason why these local invariant lines would match globally. If this mismatch persists till the limit, then we expect at the limit some $\Psi_0^\infty$ that does  not commute with $T_{\frac12}$.

\clearpage
\bibliographystyle{plain}
\bibliography{ref}

\begin{thebibliography}{10}

\bibitem{germeResonant}
C.~Christopher and C.~Rousseau.
\newblock Modulus of analytic classification for the generic unfolding of a
  codimension 1 resonant diffeomorphism or resonant saddle.
\newblock {\em Ann. Inst. Fourier (Grenoble)}, 57(1):301--360, 2007.

\bibitem{realisation}
C.~{Christopher} and C.~{Rousseau}.
\newblock The moduli space of germs of generic families of analytic
  diffeomorphisms unfolding a parabolic fixed point.
\newblock {\em International Mathematics Research Notices}, 2014(9):2494--2558,
  2014.

\bibitem{E}
J.~\'{E}calle.
\newblock {\em Les fonctions r\'{e}surgentes. {T}ome {III}}, volume~85 of {\em
  Publications Math\'{e}matiques d'Orsay}.
\newblock Universit\'{e} de Paris-Sud, D\'{e}partement de Math\'{e}matiques,
  Orsay, 1985.

\bibitem{GR}
J.~Godin and C.~Rousseau.
\newblock Analytic classification of germs of parabolic antoholomorphic
  diffeomorphisms of codimension k, 2020, arXiv:2001.06428.

\bibitem{multiNotPath}
J.H. Hubbard and D.~Schleicher.
\newblock Multicorns are not path connected.
\newblock In {\em Frontiers in complex dynamics}, volume~51 of {\em Princeton
  Math. Ser.}, pages 73--102. Princeton Univ. Press, Princeton, NJ, 2014.

\bibitem{nonlinear}
Y.S. Ilyashenko.
\newblock Nonlinear {S}tokes phenomena.
\newblock In {\em Nonlinear {S}tokes phenomena}, volume~14 of {\em Adv. Soviet
  Math.}, pages 1--55. Amer. Math. Soc., Providence, RI, 1993.

\bibitem{lecturesDiffEq}
Y.S. Ilyashenko and S.~Yakovenko.
\newblock {\em Lectures on Analytic Differential Equations}.
\newblock Graduate studies in mathematics. American Mathematical Society, 2008.

\bibitem{nonlanding}
H.~Inou and S.~Mukherjee.
\newblock Non-landing parameter rays of the multicorns.
\newblock {\em Invent. Math.}, 204(3):869--893, 2016.

\bibitem{germeDeploie}
P.~Marde{\v s}i{\'c}, R.~Roussarie, and C.~Rousseau.
\newblock Modulus of analytic classification for unfoldings of generic
  parabolic diffeomorphisms.
\newblock {\em Moscow Mathematical Journal}, 4(2):455--502, 2004.

\bibitem{multiII}
S.~Mukherjee, S.~Nakane, and D.~Schleicher.
\newblock On multicorns and unicorns {II}: bifurcations in spaces of
  antiholomorphic polynomials.
\newblock {\em Ergodic Theory Dynam. Systems}, 37(3):859--899, 2017.

\bibitem{multiI}
S.~Nakane and D.~Schleicher.
\newblock On multicorns and unicorns. {I}. {A}ntiholomorphic dynamics,
  hyperbolic components and real cubic polynomials.
\newblock {\em Internat. J. Bifur. Chaos Appl. Sci. Engrg.}, 13(10):2825--2844,
  2003.

\bibitem{Range}
R.M. Range.
\newblock {\em Holomorphic functions and integral representations in several
  complex variables}.
\newblock Graduate texts in mathematics ; 108. Springer-Verlag, New York, 1986.

\bibitem{R2}
C.~Rousseau.
\newblock {The moduli space of germs of generic families of analytic
  diffeomorphisms unfolding of a codimension one resonant diffeomorphism or
  resonant saddle}.
\newblock {\em {J. Differential Equations}}, 248(7):1745--1825, 2010.

\bibitem{R1}
C.~Rousseau.
\newblock {Analytic moduli for unfoldings of germs of generic analytic
  diffeomorphisms with a codimension k parabolic point}.
\newblock {\em {Ergodic Theory Dynam. Systems}}, 35(1):274--292, 2015.

\bibitem{V}
S.~M. {Voronin}.
\newblock {Analytic classification of germs of conformal mappings \((C,O) \to
  (C,O)\) with identity linear part}.
\newblock {\em {Funct. Anal. Appl.}}, 15:1--13, 1981.

\end{thebibliography}

\end{document}